\documentclass[10pt,reqno]{amsart}

\usepackage[a4paper,left=35mm,right=35mm,top=30mm,bottom=30mm,marginpar=25mm]{geometry}

\usepackage[utf8]{inputenc}
\usepackage[T1]{fontenc}
\usepackage{textcomp}
\usepackage{dsfont}
\usepackage{amssymb}
\usepackage{amsmath}
\usepackage{amsthm}
\usepackage{mathtools}
\mathtoolsset{showonlyrefs=true}

\usepackage{graphicx}
\usepackage{xcolor}
\usepackage{ifthen}
\usepackage{enumitem}
\usepackage[nocompress, space]{cite}
\usepackage[displaymath,mathlines]{lineno}

\DeclareMathOperator{\tr}{tr}

\DeclareMathOperator{\divergence}{div}

\newcommand{\Rr}{\mathbb{R}}

\newcommand{\cws}{\overset{*}{\rightharpoonup}}

\newcommand{\dd}{\,\mathrm{d}}
\newcommand{\dx}{\,\mathrm{d}x}
\newcommand{\dy}{\,\mathrm{d}y}
\newcommand{\dt}{\,\mathrm{d}t}
\newcommand{\ds}{\,\mathrm{d}s}

\usepackage{microtype}

\renewcommand{\leq}{\leqslant}
\renewcommand{\geq}{\geqslant}

\numberwithin{equation}{section}

\theoremstyle{plain}
\newtheorem{theorem}{Theorem}[section]
\newtheorem{lemma}[theorem]{Lemma}
\newtheorem{corollary}[theorem]{Corollary}
\newtheorem{proposition}[theorem]{Proposition}

\newtheorem{hyp}{Assumption}

\theoremstyle{definition}
\newtheorem{definition}[theorem]{Definition}
\newtheorem{problem}{Problem}

\theoremstyle{remark}
\newtheorem{remark}[theorem]{Remark}

\allowdisplaybreaks

\usepackage{hyperref}
\hypersetup{
    colorlinks=true,
    linkcolor=blue,
    citecolor=blue,
    urlcolor=blue
}

\usepackage[final]{changes} %draft vs final
\definechangesauthor[name={Diogo}, color=blue]{DG}

\makeindex

\begin{document}

\title[Ranking Mean-Field Planning Games]{Ranking Mean-Field Planning Games}

\author{Ali Almadeh}
\address{King Abdullah University of Science and Technology (KAUST), CEMSE Division, Thuwal 23955-6900, Saudi Arabia.}
\email{ali.almadeh@kaust.edu.sa}

\author{Tigran Bakaryan}
\address{Institute of Mathematics NAS RA, Center for Scientific Innovation and Education, Yerevan State University, Yerevan, Armenia.}
\email{tigran.bakaryan@csie.am}

\author{Diogo Gomes}
\address{King Abdullah University of Science and Technology (KAUST), CEMSE Division, Thuwal 23955-6900, Saudi Arabia.}
\email{diogo.gomes@kaust.edu.sa}

\author{Melih \"{U}\c{c}er}
\address{King Abdullah University of Science and Technology (KAUST), CEMSE Division, Thuwal 23955-6900, Saudi Arabia.}
\email{melih.ucer@kaust.edu.sa}



\thanks{The research was supported by KAUST baseline funds. The work of T.B. was supported by the Higher Education and Science Committee of the MESCS RA under Research Project No.~24IRF-1A001.}

\date{\today}
\begin{abstract}
This paper studies a one-dimensional Mean-Field Planning (MFP) system with a non-local, rank-based coupling. Using a potential formulation, we rewrite the system as an associated scalar partial differential equation. We prove an equivalence between classical solutions to the ranking MFP system with positive density and classical solutions to the associated potential problem, and we derive explicit reconstruction formulas.
We then identify a monotonicity structure in the associated operator, which, under strict convexity assumptions, yields uniqueness of classical solutions to the associated problem and, hence, uniqueness of the ranking MFP system up to an additive constant in the value function. Finally, under superlinear growth assumptions, we exploit monotonicity to address existence in a low-regularity setting. By formulating a variational inequality for a $q$-Laplacian regularized operator, we apply Minty’s method to establish the existence of weak solutions in the space of functions of bounded variation for a relaxed potential formulation. 
\end{abstract}

\maketitle
%\tableofcontents

\section{Introduction}\label{introd}

Mean-Field Game (MFG) theory investigates Nash equilibria in populations of infinitesimal agents where individual costs and dynamics are influenced by the aggregate distribution of states. These models were introduced independently through two distinct perspectives. The PDE-based approach, presented in abridged form in \cite{ll1, ll2} and, later, in more detail in \cite{lasryMeanFieldGames2007},  
characterizes equilibria via a system of coupled partial differential equations. Simultaneously, the Nash Certainty Equivalence (NCE) principle was developed for large-population stochastic dynamic games \cite{huangLargePopulationStochastic2006, Caines2}.
Existence, uniqueness, and regularity for these systems are nontrivial. Classical solutions for time-dependent models have been established across various regimes, including sub-quadratic \cite{gomesTimeDependentMeanFieldGames2015} and super-quadratic Hamiltonians \cite{gomesTimedependentMeanfieldGames2016a}, as well as logarithmic nonlinearities \cite{Gomes2015b}. Further refinements include Sobolev estimates for the Hamilton--Jacobi component \cite{cpt}, which are central to the analysis of first-order models. For a detailed account of these techniques, see \cite{GPV}.

The Mean-Field Planning problem (MFP) was originally introduced in \cite{LCDF} (lectures on November 27th, December 4th-11th, 2009). Unlike standard MFGs, where a terminal cost is fixed, the MFP prescribes both the initial and terminal distributions, $m_0$ and $m_T$, while leaving the terminal value function $u(T, \cdot)$ free. In MFP problems, the value function $u$ is determined only up to an additive constant. In this framework, a central planner seeks to guide a population toward a target final state, assuming the infinitesimal agents continue to select strategies that minimize their individual costs.
The MFP reads
\begin{equation*}
\left\{\begin{array}{cl}
    &-u_{t} + H\left(x, Du\right) = f\left(x,m\right), \\
    &m_t - \divergence\left(mD_{p}H\left(x,Du\right)\right) = 0, \\
    &m\left(0,x\right) = m_0\left(x\right),\quad m\left(T,x\right) = m_T\left(x\right).
\end{array}\right.
\end{equation*}
Here, $H:\Rr^d\times\Rr^d\to \Rr$ is the Hamiltonian, and $D_{p} H(x, p)$ denotes the gradient of $H$ with respect to the momentum variable $p$.

Existence and uniqueness of solutions for MFPs with quadratic-growth Hamiltonians and monotone couplings were established in \cite{LCDF} (lectures on November 27th, December 4th-11th, 2009).
 This work was subsequently extended to second-order MFPs involving diffusion \cite{porretta}, which established the existence of weak solutions. Numerically, the MFP has been addressed through finite-difference schemes \cite{CDY} and, more recently, via particle approximations for one-dimensional models in \cite{francescoParticleApproximationOnedimensional2022}.

Research on the MFP problem has drawn extensively on optimal transport theory \cite{BB}. In particular, the MFP has been analyzed as a regularized mass-transport problem \cite{Tono2019}, investigated via optimal entropy-transport duality \cite{OrPoSa2018}, and extended to incorporate the effects of congestion \cite{MR3644590}.

Here, we use a potential formulation, originally developed in \cite{DRT2021Potential}, where the Poincar\'e lemma is applied to the continuity equation to reduce the MFP to a variational problem, eliminating the continuity equation from the system. Specifically, we consider a MFP with a rank-based non-local coupling, where the instantaneous cost depends on the agent's quantile, given by the cumulative distribution function (CDF) of the density $m$.

Non-local couplings appear widely in the MFG literature and typically have a smoothing effect \cite{cardaliaguetLongTimeAverage2013}; rank-based interactions, by contrast, introduce a less regular coupling. 
Rank-based interactions appear in stochastic models, including games of diffusion control \cite{ankirchnerMeanfieldRankingGames2024}, games with common noise \cite{bayraktarRankbasedMeanField2016}, terminal ranking models \cite{bayraktarTerminalRankingGames2020}, general mean-field competition \cite{nutzMeanFieldCompetition2017}, and principal-agent frameworks \cite{alasseurRankBasedRewardPrincipal2023}, and in games with rank and nearest-neighbor effects \cite{carmonaProbabilisticWeakFormulation2015}. 
To our knowledge, a PDE-based analysis for the deterministic first-order planning problem with rank interactions is not available in the literature.
Because of the limited regularity of the CDF coupling, we work with a weak formulation using monotone operator methods.

 The general form of a one-dimensional MFP with interactions through ranking is given by
\begin{equation*}
\begin{cases}
    &-u_{t} + H\left(x,u_{x}\right) = V\left(\int_{0}^{x} f\left(m\left(t, y\right)\right)\dy\right), \\
    &m_t - \left(mD_{p}H\left(x,u_{x}\right)\right)_{x} = 0, \\
    &m\left(0, x\right) = m_0\left(x\right),  \quad m\left(T, x\right) = m_T\left(x\right).
\end{cases}
\end{equation*}
We consider the one space dimension case with $x \in [0,1]$, set $f(m) = m$ and $V(r) = r$, and interpret the ranking term as a cumulative distribution function, leading to the system in Problem~\ref{main problem}.

\subsection*{A Motivating Model}
As motivation, we model a continuum of firms whose cost depends on their rank in the emissions distribution.
 The state of each agent $x(t) \in [0, 1]$ represents its normalized emission intensity at time $t$. Here, $x(t) = 0$ corresponds to a carbon-neutral state, while $x(t) = 1$ represents the maximum allowable emission intensity.

Each firm seeks to minimize a total cost composed of two competing factors: the operational cost of decarbonization and a market-based reputation cost determined by its ranking. Firms control the rate of change of their emissions $v(t)$, which determines the evolution of their state:
\begin{equation*}
    \dot{x}(t) = v(t), \quad x(0) = x_0.
\end{equation*}
Any adjustment to the emission level incurs costs. Reducing emissions ($v < 0$) typically requires investment in green technology, while increasing emissions ($v > 0$) may require additional emission certificates, carbon taxes, or regulatory fines. These costs are modeled by the Lagrangian $L(x,v)$, taken as the Legendre transform of the Hamiltonian $H$:
\begin{equation}\label{Legendre trans}
    L(x,v) = \sup_{p \in \mathbb{R}} \left(-vp - H(x,p)\right), \quad x \in [0,1].
\end{equation}
The competitive interaction is rank-based: the cost incurred by a firm depends on its relative standing in the population. Let $m(t,y)$ denote the density of agents at emission level $y$.
The rank of a firm at time $t$ and state $x$ is generally defined via a kernel $f$ as
\begin{equation*}
r(t, x) = \int_{0}^{x} f(m(t,y)) \dy.
\end{equation*}

Each firm seeks to select an admissible control $v(\cdot)$ that minimizes the cost functional
\begin{equation*}
    \int_{0}^{T}\left[L\left( x\left( s\right),v\left( s\right)\right) + V(r\left(s ,x\left(s\right)\right))\right]\ds + u\left(T,x\left( T\right)\right),
\end{equation*}
where $V(\cdot)$ represents the rank-dependent reputation cost.  Since $V$ is increasing, firms are driven to lower their rank (i.e., move toward $x=0$). In this work, we focus on the case of a linear ranking cost, setting $V(r) = r$ and $f(m)=m$.

In the MFP framework, a central planner (e.g., a government regulator) observes the initial distribution $m_0$ and prescribes a target distribution $m_T$ (e.g., a `Green Economy' compliance target). The terminal cost $u(T, \cdot)$ is not prescribed in advance but is instead part of the solution. It is the incentive policy required to ensure that the market equilibrium naturally evolves to $m_T$.

\subsection*{Problem Formulation}
We study the following ranking MFP system  on the domain $\Omega := (0,T) \times (0,1)$.
\begin{problem}[Ranking MFP Problem]\label{main problem}
    Let $H: [0,1] \times \mathbb{R} \to \mathbb{R}$ and $m_0, \, m_{T}:\left[0,1\right] \to \mathbb{R}^{+}$. We study the solutions $\left(u,m\right)$ of the PDE system
\begin{equation}\label{prob MFP}
\left\{
\begin{aligned}
    &-u_{t} + H\left(x,u_{x}\right) = \int_{0}^{x}m\left(t, y\right)\dy, \\
    &m_t - \left(mD_{p}H\left(x,u_{x}\right)\right)_{x} = 0,
\end{aligned}
\right. \qquad\text{in}\quad\Omega,
\end{equation}
    that satisfy $m \geq 0$ and the boundary conditions
\begin{equation}\label{boundary MFP}
\begin{cases}
     m(0, x) = m_{0}(x), & x \in [0, 1], \\
     m(T, x) = m_{T}(x), & x \in [0, 1], \\
     - m(t,x)D_{p}H(x,u_{x}) = 0, & \text{on } (0, T) \times \{0,1\}.
\end{cases}
\end{equation}
\end{problem}
The lateral boundary condition $-mD_{p}H = 0$ represents a zero-flux constraint, which ensures that the total mass of agents is conserved within the domain $[0,1]$.

\begin{remark}[On the Choice of Coupling]
The choice $f(m)=m$ and $V(r)=r$ is analytically significant because the resulting CDF coupling,
$m \mapsto \int_0^x m$, fails to be continuous as a map from probability measures equipped with the $W_1$ topology into $C([0,1])$ when limiting measures develop singular parts.
 This lack of smoothing prevents the application of standard fixed-point arguments. 
 The linear choice nevertheless yields a monotone operator, allowing a 
$BV$ existence theory.
\end{remark}

To analyze Problem~\ref{main problem}, we introduce an associated problem via a potential formulation. 
By applying the Poincar\'{e} lemma to the divergence-free field associated with the transport equation, we identify a potential $\varphi$ satisfying $\varphi_x=m$ (and, for classical solutions, $\varphi_t=mD_{p}H(x,u_x)$).
This allows us to rewrite the system as a single equation for $\varphi$, which we refer to as the Associated Problem. More precisely, the planning constraints are encoded by the boundary traces
$\varphi(0,\cdot)$ and $\varphi(T,\cdot)$.
Moreover, the zero-flux condition implies that $\varphi$ is constant in time on
$x=0,1$; together with the mass normalization and our reference potential
$\varphi^{0}$, this yields the Dirichlet boundary condition
$\varphi=\varphi^{0}$ on $\partial\Omega$ in the classical setting. The precise derivation is detailed in Section~\ref{sec:linking}. To define this problem formally, we first introduce the reference potential $\varphi^0$, which interpolates the boundary data,
\begin{equation}\label{defining phi^0}
    \varphi^{0}\left( t,x\right) = \frac{T-t}{T} \int_{0}^{x}m_{0}\left( y\right)\dy+ \frac{t}{T} \int_{0}^{x}m_{T}\left( y\right)\dy,
\end{equation}
and the perspective function $F$, defined as the partial Legendre transform
\begin{equation}\label{gen_pers_func}
F(x,j,m) = \sup_{p \in \mathbb{R}} \big(-jp - mH(x,p)\big).
\end{equation}
For a strictly positive density $m > 0$, $F$ admits the representation
\begin{equation}\label{non-sing pers func}
F(x,j,m) = L\left(x,\frac{j}{m}\right)m, \quad j \in \mathbb{R},
\end{equation}
where $L$ is the Legendre transform defined in \eqref{Legendre trans}.

Now, the Associated Problem for our Ranking MFP Problem is outlined as follows. 
\begin{problem}[Associated Problem]\label{associated problem}
    Consider the setting of Problem~\ref{main problem} and let $F$ and $\varphi^0$ be as above. We study the solutions $\varphi$ of the equation
    \begin{equation}\label{assoc prob}
            \left(F_{j}\left( x,-\varphi_{t},\varphi_{x}\right)\right)_{t} + \left(- F_{m}\left( x,-\varphi_{t},\varphi_{x}\right)\right)_{x} - \varphi_{x} = 0 \qquad \text{in} \quad \Omega,
    \end{equation}
    that satisfy $\varphi_x\geq 0$ and the boundary condition
    \begin{equation} \label{boundary assoc}
         \varphi = \varphi^{0} \quad \text{on} \quad \partial\Omega.
    \end{equation}
  \end{problem}
Having defined our two central problems, we now outline our main results and the structure of our analysis. Our work proceeds in two main parts.

In the first part, we begin with establishing a formal equivalence between the classical solutions of these two problems via
\begin{equation*}
            \left\{
            \begin{aligned}
                 m(t,x) &= \varphi_{x}(t,x), \\[2ex]
                 u(t,x) &= -\int_{0}^{x} F_{j}\Bigl(y,-\varphi_{t}(t,y),\varphi_{x}(t,y)\Bigr)\dy \\
                        &\qquad -\int_{0}^{t} F_{m}\Bigl(0,-\varphi_{t}(s,0),\varphi_{x}(s,0)\Bigr)\ds.
            \end{aligned}
            \right.
\end{equation*}
Normalizing $u(0,0)=0$ determines the time-integral term in the reconstruction formula. Then, we prove the uniqueness of classical solutions to Problem~\ref{associated problem}. Under the equivalence in Lemma~\ref{Linking Lemma}, this uniqueness result for the associated problem transfers to the ranking MFP (modulo an additive constant in $u$). Theorem~\ref{uniq. assoc. prob.} and Corollary~\ref{uniq. main prob.} below express these results.
%Second, we leverage this link to prove the uniqueness of classical solutions for both systems. Theorem~\ref{uniq. assoc. prob.} establishes that classical solutions to Problem~\ref{associated problem} are unique.
\begin{theorem}[Uniqueness]\label{uniq. assoc. prob.}
    Suppose Assumptions~\ref{assump:H} and~\ref{assump:positivity} hold. Then Problem~\ref{associated problem} admits at most one classical solution $\varphi$ with $\varphi_{x} > 0$.
\end{theorem}
%However, because of the way the problems are linked, we have the uniqueness of classical solutions to Problem~\ref{main problem} only up to an additive constant in $u$.
\begin{corollary}\label{uniq. main prob.}
    Suppose Assumptions~\ref{assump:H} and~\ref{assump:positivity} hold. Then Problem~\ref{main problem} admits at most one classical solution $(u, m)$ with $m > 0$, up to an additive constant in $u$.
\end{corollary}

In the second part, our main result is Theorem~\ref{Existence theorem} (existence of 
$BV$ weak solutions), proved via monotone operator theory and a 
$q$-Laplacian regularization of Problem~\ref{associated problem} followed by Minty’s method. We define weak solutions as follows.

\begin{definition}[Weak Solution]\label{def:weak_solution}
A function $\varphi \in BV(\Omega)$ is a weak solution to Problem~\ref{associated problem} 
if $D_x\varphi \geq 0$ in the sense of Radon measures, $\varphi(0, \cdot) = \varphi^0(0, \cdot)$, $\varphi(T, \cdot) = \varphi^0(T, \cdot)$, $\varphi(\cdot,0)\geq \varphi^0(\cdot,0) = 0$, and $\varphi(\cdot,1)\leq \varphi^0(\cdot,1) = 1$ in the sense of traces, and for every test function $\eta \in \mathcal{C}^1(\bar{\Omega})$ satisfying $\eta_x > 0$ and $\eta = \varphi^0$ on $\partial\Omega$, the following variational inequality holds
 \begin{multline*}
\int_{\Omega} \left[ -F_j(x,-\eta_t, \eta_x)\eta_t + F_m(x,-\eta_t, \eta_x)\eta_x + \eta\eta_x \right] \dd \Omega \\
- \biggl[\int_{\Omega} -F_j(x,-\eta_t, \eta_x) D_t\varphi + \int_{\Omega} F_m(x,-\eta_t, \eta_x) D_x\varphi + \int_{\Omega} \eta D_x\varphi \\
+ L(0,0)\int_0^T \varphi(t,0)\dt + \bigl(L(1,0)+1\bigr)\int_0^T (1-\varphi(t,1))\dt \biggr] \geq 0,
\end{multline*}
where $D_t\varphi$ and $D_x\varphi$ are the Radon measures that represent the distributional derivatives of~$\varphi$.
\end{definition}
We recall basic facts on $BV(\Omega)$ in Subsection~\ref{subsec:BV_pre}.
 \begin{remark}[Boundary conditions and Minty formulation]
Definition~\ref{def:weak_solution} formulates a weak solution in the sense of Minty, where the test function $\eta$ appears within the operator.
While the classical problem imposes the pointwise Dirichlet condition $\varphi=\varphi^0$ on $\partial\Omega$, the $BV$ formulation requires a relaxed approach.
On the lateral boundary $(0,T)\times\{0,1\}$ we do not impose pointwise equality; instead, the spatial boundary data is encoded via the one-sided trace constraints in Definition~\ref{def:weak_solution} together with the boundary terms in the variational inequality.
Physically, these inequalities allow a loss of mass in the open interval $(0,1)$, which corresponds to concentrations at $x=0$ and $x=1$, forming Dirac masses. Therefore, the formulation naturally captures boundary congestion.
In contrast, the initial and terminal traces at $t=0$ and $t=T$ are imposed as equalities in the sense of traces.

Moreover, the boundary terms in the variational inequality weight the boundary concentrations
$\varphi(t,0)$ and $\left(1-\varphi(t,1)\right)$ by the running cost at zero velocity, $L(x,0)$ (and the additional ranking contribution at $x=1$).
Indeed, since $\eta=\varphi^0$ on $\partial\Omega$, we have $\eta_t=\varphi_t^0$ on $(0,T)\times\{0,1\}$; and because $\varphi^0(t,0)=0$ and $\varphi^0(t,1)=1$ for all $t$, it follows that $\eta_t(t,0)=\eta_t(t,1)=0$.
Using $F_m(x,0,m)=L(x,0)$ for $m>0$, we obtain
\[
F_m(0,-\eta_t(t,0),\eta_x(t,0))=L(0,0),
\qquad
F_m(1,-\eta_t(t,1),\eta_x(t,1))=L(1,0).
\]
\end{remark}
\begin{theorem}[Existence of Weak Solutions]\label{Existence theorem}
Suppose Assumptions~\ref{assump:H},~\ref{assump:positivity},~\ref{assump:L-growth},~\ref{assump:low_bnd_L} hold. Then there exists a weak solution $\varphi$ to Problem~\ref{associated problem} in the sense of Definition~\ref{def:weak_solution}.
\end{theorem}
To prove Theorem~\ref{Existence theorem}, we use a framework based on monotone operators in reflexive Banach spaces.
In Hilbert spaces, this approach has been applied to prove the existence of weak solutions for both time-dependent \cite{FeGoTa21} and stationary MFGs \cite{FG2, FGT1}. Recently, it was unified in a Banach-space framework for stationary MFGs, including models with congestion and minimal growth Hamiltonians \cite{ferreiraSolvingMeanFieldGames2025}. In the present work, we use a
$q$-Laplacian regularization of the operator associated with Problem~\ref{associated problem}, as given in~\eqref{reg. operator}. Using the Abstract Existence Theorem for monotone coercive operators, Theorem~\ref{abstract existence theorem}, we show the existence of solutions to the regularized problem. Finally, we apply Minty's method to pass to the limit, yielding a weak solution in $BV(\Omega)$, see Section~\ref{existence}.

\section{Classical Solutions}
In this section, we consider the case where solutions are sufficiently smooth. We prove the equivalence between the Ranking MFP and the Associated Problem and deduce uniqueness of classical solutions.
\subsection{Assumptions and Preliminary Identities}\label{assumptionsI}
We state the structural assumptions on the Hamiltonian and the boundary data which are used throughout the paper; additional growth hypotheses required for the existence theory are introduced in subsection~\ref{pre and assump}.
\begin{hyp}[Hamiltonian Properties]\label{assump:H}
    The Hamiltonian $H$ satisfies the following regularity conditions:
    \begin{equation*}
        H \in \mathcal{C}^1([0,1]\times \mathbb{R}) \quad \text{and} \quad D_{p}H \in \mathcal{C}^1([0,1]\times \mathbb{R}).
    \end{equation*}
    Furthermore, $H$ is coercive and strictly convex in $p$. That is, for all $x \in [0,1]$,
    \begin{equation*}
        \lim_{|p|\to\infty} \frac{H(x,p)}{|p|} = +\infty \quad \text{and} \quad D_{pp} H(x,p) > 0.
    \end{equation*}
\end{hyp}
For brevity, we restrict our analysis to the autonomous case $H(x,p)$; 
however, the arguments extend to time-dependent Hamiltonians $H(t,x,p)$ satisfying Assumption~\ref{assump:H} uniformly.
\begin{remark}[Properties of $L$ and $F$]\label{rem:Lagrangian}
    The Lagrangian $L(x,v)$ is defined via the Legendre transform~\eqref{Legendre trans}. Under Assumption~\ref{assump:H}, standard convex analysis implies that $L$ is strictly convex in $v$ and inherits the regularity of $H$, specifically:
    \begin{equation*}
        L \in \mathcal{C}^1([0,1]\times \mathbb{R}) \quad \text{and} \quad D_{v}L \in \mathcal{C}^1([0,1]\times \mathbb{R}).
    \end{equation*}
    We will use the following identities:
    \begin{enumerate}
        \item \textbf{Legendre Transform Identities:} At points where the Legendre duality holds (i.e., $D_{p} H(x,p) = - v$ and $D_{v} L(x, v) = -p$), we have
        \begin{equation}\label{compos D_{v}L and D_{p}H}
        \begin{aligned}
             D_{p} H\left(x , -D_{v} L\left(x, v\right)\right) &= - v, \\
             H\left(x , -D_{v} L\left(x, v\right)\right) &= vD_{v}L(x,v)-L(x,v).
        \end{aligned}
        \end{equation}
        \item \textbf{Perspective Function Derivatives:} For $F$ given in~\eqref{non-sing pers func}, the derivatives are
        \begin{equation}\label{eq:pers-derivatives}
        \begin{aligned}
            F_{j} &= D_{v}L\left( x,\frac{j}{m}\right), \\
            F_{m} &= -D_{v}L\left( x,\frac{j}{m}\right) \cdot \left(\frac{j}{m}\right) + L\left( x,\frac{j}{m}\right).
        \end{aligned}
        \end{equation}
    \end{enumerate}
    Consequently, $F$, $F_j$, and $F_m$ are in $\mathcal{C}^1([0,1]\times\mathbb{R}\times\mathbb{R}^+)$ because of~\eqref{non-sing pers func} and~\eqref{eq:pers-derivatives}. Furthermore, recalling the definition in~\eqref{gen_pers_func}, we observe that $F(x,\cdot,\cdot)$ is the supremum of a family of linear functions; therefore, it is jointly convex for any fixed $x\in[0,1]$.
\end{remark}

\begin{lemma}\label{lem:pers-mu}
    Suppose Assumption~\ref{assump:H} holds. For any $x\in[0,1]$ and any $(j,m), (j_0,m_0)\in\mathbb{R}\times\mathbb{R}^+$, we have
    \[F_j(x,j,m)j_0 + F_m(x,j,m)m_0 \leq F(x,j_0,m_0).\]
\end{lemma}

\begin{proof}
    By~\eqref{eq:pers-derivatives}, we have
    \[F_j(x,j,m)j_0 + F_m(x,j,m)m_0 = m_0\left(D_{v}L\left(x,\frac{j}{m}\right)\left(\frac{j_0}{m_0}-\frac{j}{m}\right)+L\left(x,\frac{j}{m}\right)\right).\]
    By the convexity of $L(x,\cdot)$ in Remark~\ref{rem:Lagrangian}, we have
    \[D_{v}L\left(x,\frac{j}{m}\right)\left(\frac{j_0}{m_0}-\frac{j}{m}\right)+L\left(x,\frac{j}{m}\right) \leq L\left(x,\frac{j_0}{m_0}\right).\]
    Therefore,
    \[F_j(x,j,m)j_0 + F_m(x,j,m)m_0 \leq m_0L\left(x,\frac{j_0}{m_0}\right) = F(x,j_0,m_0).\qedhere\]
\end{proof}

Finally, we impose the following conditions on the boundary data.

\begin{hyp}[Positivity of Boundary Data]\label{assump:positivity}
    We assume the initial and terminal distributions $m_0, m_{T} \in \mathcal{C}^{1}\left(\left[0,1\right]\right)$ are strictly positive probability densities. That is, $m_0(x) > 0$, $m_T(x) > 0$ for all $x \in [0,1]$, and $\int_{0}^{1} m_{0}\left(x\right)\dx = 1$, $\int_{0}^{1} m_{T}\left(x\right)\dx =1$.
\end{hyp}

\subsection{The Linking Lemma}\label{sec:linking}
In this subsection, we prove Lemma~\ref{Linking Lemma}, establishing the equivalence between the Ranking MFP Problem and the Associated Problem for classical solutions.
\begin{lemma}[Linking Lemma]\label{Linking Lemma}
    Consider the setting of Problem~\ref{main problem} and suppose Assumptions~\ref{assump:H} and~\ref{assump:positivity} hold.
    \begin{enumerate}
        \item Let $\varphi \in \mathcal{C}^2(\bar{\Omega})$ be a classical solution to Problem~\ref{associated problem}, with $\varphi_{x} > 0$. Then, the pair $(u, m)$ defined by
        \begin{equation}\label{solutions}
            \left\{
            \begin{aligned}
                 m(t,x) &= \varphi_{x}(t,x), \\[2ex]
                 u(t,x) &= -\int_{0}^{x} F_{j}\Bigl(y,-\varphi_{t}(t,y),\varphi_{x}(t,y)\Bigr)\dy \\
                        &\qquad -\int_{0}^{t} F_{m}\Bigl(0,-\varphi_{t}(s,0),\varphi_{x}(s,0)\Bigr)\ds
            \end{aligned}
            \right.
        \end{equation}
        is a classical solution to Problem~\ref{main problem}.

        \item Conversely, let $(u,m) \in \mathcal{C}^2(\bar{\Omega}) \times \mathcal{C}^1(\bar{\Omega})$ be a classical solution to Problem~\ref{main problem} with $m>0$ and $u(0,0) = 0$.  Then there exists a potential function $\varphi$ satisfying~\eqref{solutions} that solves Problem~\ref{associated problem}.
    \end{enumerate}
\end{lemma}
\begin{proof}
First, we show the implication \textbf{Associated Problem $\implies$ Ranking MFP Problem.}
Differentiating $u$ in~\eqref{solutions} with respect to $x$ and using the identity for $F_j$ from~\eqref{eq:pers-derivatives}, we obtain
    \begin{equation}\label{eq:ux.from.phi}
        u_{x} = - F_{j}\left( x,-\varphi_{t},\varphi_{x}\right) = - D_{v}L\left(x, \frac{-\varphi_{t}}{\varphi_{x}}\right).
    \end{equation}
    Next, we substitute $u_x$ into the terms $mD_{p} H(x, u_x)$ and $H(x, u_x)$. By applying the Legendre transform identities~\eqref{compos D_{v}L and D_{p}H} with velocity $v = \frac{-\varphi_t}{\varphi_x}$, and recalling $m = \varphi_x$ from~\eqref{solutions}, we deduce
    \begin{equation}\label{eq:legendre_perspective_relations}
    \begin{aligned}
        mD_{p} H\left(x , u_{x}\right) &= \varphi_x D_{p} H\left(x , -D_{v} L\left(x, \frac{-\varphi_{t}}{\varphi_{x}}\right)\right) = \varphi_x \left(\frac{\varphi_t}{\varphi_x}\right) = \varphi_t, \\
        H\left(x, u_{x}\right) &= H\left(x , -D_{v} L\left(x, \frac{-\varphi_{t}}{\varphi_{x}}\right)\right) \\
        &= \left(\frac{-\varphi_{t}}{\varphi_{x}}\right) D_{v}L\left(x, \frac{-\varphi_{t}}{\varphi_{x}}\right) -L\left(x, \frac{-\varphi_{t}}{\varphi_{x}}\right) = - F_{m}\left( x,-\varphi_{t},\varphi_{x}\right).
    \end{aligned}
    \end{equation}
    Note that the final equality for $H(x, u_x)$ matches the identity for $-F_m$ given in~\eqref{eq:pers-derivatives}.

Next, using the first equation in~\eqref{solutions} and the first identity in~\eqref{eq:legendre_perspective_relations}, we obtain
    \begin{equation*}
       0 = \left(\varphi_{x}\right)_{t} - \left(\varphi_{t}\right)_{x} = m_{t} - \left(mD_{p} H\left(x , u_{x}\right)\right)_{x}.
    \end{equation*}
    That is, the transport equation~\eqref{prob MFP} holds for $(u,m)$.

Now, for the HJ-equation, observe that from~\eqref{solutions}
\begin{equation}\label{u_t in tilde{F}_j and tilde{F}_m}
    u_t(t,x) = \left(-\int_0^x \frac{\partial}{\partial t} F_j(y, -\varphi_t(t,y), \varphi_x(t,y)) \dy \right) - F_m(0, -\varphi_t(t,0), \varphi_x(t,0)).
\end{equation}

Using~\eqref{assoc prob}, we get
\begin{equation*}
 \frac{\partial}{\partial t} F_j(y, -\varphi_t(t,y), \varphi_x(t,y)) = \frac{\partial}{\partial y} F_m(y, -\varphi_t(t,y), \varphi_x(t,y)) + \varphi_x(t,y).
\end{equation*}
Substituting this into~\eqref{u_t in tilde{F}_j and tilde{F}_m}, using~\eqref{eq:legendre_perspective_relations} and~\eqref{solutions} yields
\begin{align*}
 u_t(t,x) &= -\int_0^x \frac{\partial}{\partial y} F_m(y, -\varphi_t(t,y), \varphi_x(t,y)) \dy \\
 & \qquad - \int_0^x \varphi_x(t,y) \dy - F_m(0, -\varphi_t(t,0), \varphi_x(t,0)) \\
    &= - F_m(x, -\varphi_t(t,x), \varphi_x(t,x)) - \int_0^x \varphi_x(t,y) \dy \\
    &= H(x, u_x) - \int_0^x m(t,y) \dy.
\end{align*}

It follows that $(u,m) \in \mathcal{C}^2(\bar{\Omega})\times \mathcal{C}^1(\bar{\Omega})$ solves Problem~\ref{main problem}.

To conclude the implication, we need to examine the boundary conditions.  
Assume that we have the boundary condition $\varphi = \varphi^{0}$ on $\partial\Omega$ as in~\eqref{boundary assoc}. This implies
       \begin{align*}
           \varphi\left(0,x\right) &= \varphi^{0}\left(0,x\right) = \int_{0}^{x}m_{0}\left( y\right)\dy, \\
           \varphi\left(T,x\right) &= \varphi^{0}\left(T,x\right)= \int_{0}^{x}m_{T}\left( y\right)\dy.
       \end{align*}
       Consequently,
       \begin{align*}
           m\left(0,x\right) &= \varphi_{x}\left(0,x\right) = m_{0}\left(x\right),\\
           m\left(T,x\right) &= \varphi_{x}\left(T,x\right) = m_{T}\left(x\right).
       \end{align*}
       Furthermore,
       \begin{align*}
           \varphi\left(t,0\right) &= \varphi^{0}\left(t,0\right) = 0, \\
           \varphi\left(t,1\right) &= \varphi^{0}\left(t,1\right) = 1.
       \end{align*}
       Hence, we have
       \begin{equation*}
           \varphi_{t}\left(t,0\right) = \varphi_{t}\left(t,1\right) = 0 \implies -mD_{p}H\left(x, u_{x}\right) = 0 \quad \text{on } (0, T) \times \{0, 1\}.
       \end{equation*}
   
Now we prove the opposite implication, \textbf{Ranking MFP Problem $\implies$ Associated Problem.} Suppose that $\left( u,m\right) \in \mathcal{C}^2(\bar{\Omega})\times \mathcal{C}^1(\bar{\Omega})$ solves Problem~\ref{main problem} and satisfies $m > 0$. We can rewrite the transport equation $m_t - \left( mD_{p}H\left( x,u_{x}\right)\right)_{x} = 0$ in the following form
\begin{equation*}
    \divergence_{\left( t,x\right)}\left( m, -mD_{p}H\left( x,u_{x}\right)\right) = 0.
\end{equation*}
Since $\Omega$ is simply connected, this satisfies the hypotheses of the Poincar\'{e} lemma. Therefore, there exists a unique function $\varphi: \Omega \rightarrow \mathbb{R}$ such that
\begin{equation}\label{potential function}
    \begin{cases}
        \varphi_{x} = m, \\
        \varphi_{t} = mD_{p}H\left( x,u_{x}\right),\\
        \varphi(0,0) = 0.
    \end{cases}
\end{equation}
 Since $\left( u,m\right) \in \mathcal{C}^2(\bar{\Omega})\times \mathcal{C}^1(\bar{\Omega})$, by the regularity of $H$ in Assumption~\ref{assump:H}, we have $\varphi \in \mathcal{C}^2(\bar{\Omega})$.

From~\eqref{potential function}, we have $\frac{-\varphi_{t}}{\varphi_{x}} = -D_{p}H(x, u_x)$. Using this in~\eqref{compos D_{v}L and D_{p}H} yields the relations in~\eqref{eq:ux.from.phi} and~\eqref{eq:legendre_perspective_relations}.
Differentiating the HJ-equation in~\eqref{prob MFP} with respect to $x$, and using~\eqref{eq:ux.from.phi} and~\eqref{eq:legendre_perspective_relations}, we obtain
\begin{equation}\label{persp_HJ}
    \left( F_{j}\left( x,-\varphi_{t},\varphi_{x}\right)\right)_{t} + \left(- F_{m}\left( x,-\varphi_{t},\varphi_{x}\right)\right)_{x} - \varphi_{x} = 0,
\end{equation}
which is~\eqref{assoc prob}.
Now, by the construction of $\varphi$, the first equation in~\eqref{solutions} is satisfied. For the second equation, integrating~\eqref{eq:ux.from.phi} with respect to $x$ yields
\begin{equation}\label{eq:u_from_phi}
    u(t,x) = -\int_{0}^{x} F_{j}\Bigl(y,-\varphi_{t}(t,y),\varphi_{x}(t,y)\Bigr)\dy + u\left( t, 0 \right).
\end{equation}
Evaluating the Hamilton--Jacobi equation in~\eqref{prob MFP} at $x =0$ using~\eqref{eq:legendre_perspective_relations}, and integrating the result with respect to $t$, we obtain
\begin{equation}\label{u(t,0)_from_phi}
    u\left(t, 0 \right) = -\int_{0}^{t} F_{m}\Bigl(0,-\varphi_{t}(s,0),\varphi_{x}(s,0)\Bigr)\ds.
\end{equation}
Combining~\eqref{eq:u_from_phi} and~\eqref{u(t,0)_from_phi} justifies the second equation in~\eqref{solutions}. Thus, $\varphi$ satisfies~\eqref{solutions}, as claimed.
Finally, suppose that
 $(u,m)$ satisfies the boundary and no-flux conditions in~\eqref{boundary MFP}. Then, by the definition of $\varphi$ in~\eqref{potential function}, we have
\begin{equation*}
    \varphi_{x} \left(0,x\right) = m_{0}\left(x\right) \quad \text{and} \quad \varphi_{x} \left(T,x\right) = m_{T}\left(x\right).
\end{equation*}
Furthermore, the no-flux condition, $-mD_{p}H = 0$, on the spatial boundaries implies
\begin{equation*}
    \varphi_{t} \left(t,0\right) = \varphi_{t} \left(t,1\right) = 0,
\end{equation*}
for all $t \in [0,T]$. This implies
\begin{equation*}
   \varphi\left(t,0\right) = \varphi\left(0,0\right) = 0  \quad \text{and} \quad \varphi\left(t,1\right) = \varphi\left(0,1\right).
\end{equation*}
Therefore, by the Fundamental Theorem of Calculus, we get
\begin{equation*}
     \varphi(0,1) = \varphi(0,0) +\int_{0}^{1} \varphi_{x} \left(0,x\right)\dx = \int_{0}^{1} m_{0} \left(x\right)\dx =1,
\end{equation*}
implying $\varphi(t,1) = 1$.
Combining these with the spatial derivatives, we see that
\begin{equation*}
    \varphi\left(0,x\right) = \int_{0}^{x} m_{0} \left(y\right)\dy \quad \text{and} \quad \varphi\left(T,x\right) = \int_{0}^{x} m_{T} \left(y\right)\dy.
 \end{equation*}
 Thus, $\varphi = \varphi^{0}$ on $\partial\Omega$ as desired.

\end{proof}

\subsection{Uniqueness of Classical Solutions}\label{uniqueness}
In this subsection, we prove that the Associated Problem admits at most one classical solution, implying that the Ranking MFP Problem admits at most one classical solution up to an additive constant in~$u$. The argument relies on the strict convexity of the Lagrangian, which is inherited from the Hamiltonian (see Remark~\ref{rem:Lagrangian}).
\begin{proof}[Proof of Theorem~\ref{uniq. assoc. prob.}]
Suppose that $\varphi^1$ and $\varphi^2$ are two distinct classical solutions to Problem~\ref{associated problem}.
Since both $\varphi^1$ and $\varphi^2$ are solutions to the PDE in Problem~\ref{associated problem}, they satisfy, omitting arguments for brevity,
\[ \left( F_{j}^i \right)_{t} - \left( F_{m}^i \right)_{x} -\varphi_{x}^i = 0, \]
where $F^i = F(x, -\varphi^i_t, \varphi^i_x)$. Subtracting the equation for $i=1$ from that for $i=2$, testing against $\varphi^2 - \varphi^1$, and integrating over $\Omega$ yields, after integration by parts,
\[
\int_\Omega \left[ -(F_j^2 - F_j^1)(\varphi_t^2 - \varphi_t^1) + (F_m^2 - F_m^1)(\varphi_x^2 - \varphi_x^1) - (\varphi_{x}^2-\varphi_{x}^1)(\varphi^2-\varphi^1) \right] \dd \Omega = 0.
\]
The third term vanishes upon integration by parts due to the boundary conditions. To simplify the notation, let $v^{i} = \frac{-\varphi^{i}_{t}}{\varphi^{i}_{x}}$ and $L^{i} = L(x, v^i)$. Therefore, by~\eqref{eq:pers-derivatives}, we have
\begin{align*}
    0 &= \int_\Omega \left[ -(F_j^2 - F_j^1)(\varphi_t^2 - \varphi_t^1) + (F_m^2 - F_m^1) (\varphi_x^2 - \varphi_x^1)\right] \dd \Omega \\
    &= \int_{\Omega} \Big[-\left(D_{v}L^{2} - D_{v}L^{1}\right)\left(\varphi_t^2 - \varphi_t^1\right)+ \left( -v^2D_{v}L^2 + L^2 + v^1D_{v}L^1 - L^1\right)\left(\varphi_x^2 - \varphi_x^1\right)\Big] \dd \Omega \\
    &= \int_{\Omega} \Big[ \varphi_{x}^{1} \left(L^1 - L^2 - D_{v}L^2(v^1-v^2)\right)+ \varphi_{x}^{2}\left( L^2 - L^1 - D_{v}L^1(v^2-v^1)\right) \Big] \dd \Omega.
\end{align*}
By the strict convexity of $L(x, \cdot)$, we have the inequality $L(a) > L(b) + D_{v}L(b) \cdot (a-b)$ for all $a\neq b$. This means that the two terms in the final integrand are both non-negative:
\begin{equation*}
\varphi_{x}^{1} \underbrace{\left( L^1 - L^2 - D_{v}L^2(v^1-v^2)\right)}_{\geq 0} + \varphi_{x}^{2}\underbrace{\left( L^2 - L^1 - D_{v}L^1(v^2-v^1)\right)}_{\geq 0}.
\end{equation*}
Since $\varphi^{i}_{x} > 0$ for classical solutions, both terms in the integrand must be identically zero. Due to the strict convexity of $L$, we must have
$$ v^1(t,x) = v^2(t,x) \quad \text{in} \quad \Omega. $$
Since $v^1 = v^2$, it follows immediately, 
by~\eqref{eq:pers-derivatives}, that $F_j^1 = F_j^2$ and $F_m^1 = F_m^2$.

Therefore, the PDE for the difference $\varphi^1 - \varphi^2$ simplifies to
\[(\varphi_x^1 - \varphi_x^2) = 0 \quad \text{i.e.,} \quad \varphi_x^1 = \varphi_x^2. \]
We now have that $\varphi_x^1 = \varphi_x^2$ and $v^1=v^2$, so $\varphi_t^1 = \varphi_t^2$.
Therefore, the gradients are identical, $\nabla \varphi^1 = \nabla \varphi^2$. This means $\varphi^1 - \varphi^2$ is a constant. As both $\varphi^1$ and $\varphi^2$ satisfy the same Dirichlet boundary condition $\varphi = \varphi^0$ on $\partial\Omega$, it follows that $\varphi^1 = \varphi^2$. This contradicts the assumption that the solutions were distinct. Hence, the classical solution is unique.
\end{proof}
As a corollary of Lemma~\ref{Linking Lemma} and Theorem~\ref{uniq. assoc. prob.}, we have the uniqueness for the original Ranking MFP system, Problem~\ref{main problem}.
\begin{proof}[Proof of Corollary~\ref{uniq. main prob.}]
     Let $(u^1, m^1)$ and $(u^2, m^2)$ be solutions to Problem~\ref{main problem} with $m^1, m^2 > 0$. For any solution $(u,m)$, we have that $(u-u(0,0), m)$ is also a solution. By Lemma~\ref{Linking Lemma} and Theorem~\ref{uniq. assoc. prob.}, the solution pairs $(u^1-u^1(0,0), m^1)$ and $(u^2-u^2(0,0), m^2)$ both correspond to the same potential $\varphi$. Therefore, $m^1 = m^2$ and $u^1-u^2 = u^1(0,0)-u^2(0,0)$ is a constant.
\end{proof}

\section{Weak Solutions}
This section develops the variational inequality framework associated with Problem~\ref{associated problem} and proves the existence of weak solutions in the sense of Definition~\ref{def:weak_solution}.

We impose $\varphi = \varphi^{0}$ on $\partial\Omega$. We set
\begin{equation}\label{homogenization}
    \psi = \varphi - \varphi^0,
\end{equation}
so that $\psi=0$ on $\partial\Omega$. For $0<\varepsilon<1$, we add a $q$-Laplacian-type regularization to obtain a coercive problem
with solution $\psi^\varepsilon$; we then pass to the limit as $\varepsilon\to0$ and show that
the limit yields a $BV$ weak solution of Problem~\ref{associated problem}
(in the sense of Definition~\ref{def:weak_solution}).

\subsection{Preliminaries and Assumptions}\label{pre and assump}
We recall the key notions from monotone operator theory and functions of bounded variation used in the sequel and present the minimal set of assumptions for our analysis.

\subsubsection{Monotone Operator Theory}\label{preliminaries}
Let $X$ be a reflexive Banach space with dual $X'$ and let $\mathbb{K} \subset X$ be a non-empty convex set. We denote by $\langle\cdot,\cdot\rangle$ the duality pairing between $X'$ and $X$.
\begin{definition}[Properties of an Operator]\label{def:operator_properties}
    An operator $\mathcal{A}: \mathbb{K} \rightarrow X'$ is 
    \begin{enumerate}
        \item \textbf{Monotone} if for all $u, v \in \mathbb{K}$,
            $$ \left\langle \mathcal{A}u - \mathcal{A}v, u - v \right\rangle \geq 0. $$
        \item \textbf{Coercive} if there exists $\bar{u}\in\mathbb{K}$ such that
            \begin{equation} 
            \label{coercive}
            \frac{\langle\mathcal{A} u, u - \bar{u} \rangle}{\lVert u \rVert_{X}} \rightarrow \infty, \end{equation}
            as $\|u\|_X \to \infty$ for $u \in \mathbb{K}$.
        \item \textbf{Hemicontinuous} if for all $u, v \in \mathbb{K}$ and $w \in X$, the mapping 
            $$ s \longmapsto \left\langle \mathcal{A}\left(su + \left(1-s\right) v\right), w \right\rangle $$
        is continuous on the interval $\left[0,1\right]$.
    \end{enumerate}
\end{definition}

\begin{lemma}\label{coercivity_lemma}
    Let $\mathcal{A}: \mathbb{K} \rightarrow X'$ be given by $\mathcal{A}= \mathcal{A}_{1} + \mathcal{A}_{2}$, where $\mathcal{A}_{i}$ is a monotone operator for $i = 1,2$. If one of the operators, say $\mathcal{A}_{2}$, is coercive, then $\mathcal{A}$ is coercive.
\end{lemma}
\begin{proof}
    Let $\mathcal{A}_{2}$ be coercive and let $\bar{u}\in \mathbb{K}$ be as in \eqref{coercive}. By the monotonicity of $\mathcal{A}_1$, we have for any $u \in \mathbb{K}$
    \begin{equation*}
        \langle \mathcal{A}_1 u - \mathcal{A}_1(\bar{u}), u -\bar{u} \rangle \geq 0.
    \end{equation*}
    Rearranging and using the properties of the duality pairing yields a lower bound for the first term
    \begin{equation*}
        \langle \mathcal{A}_1 u, u - \bar{u} \rangle \geq \langle \mathcal{A}_1(\bar{u}), u- \bar{u} \rangle \geq -\|\mathcal{A}_1( \bar{u})\|_{X'} (\|u\|_X + \|\bar{u}\|_X).
    \end{equation*}
    Now, consider the full operator $\mathcal{A}$. We have
    \begin{equation*}
        \frac{\langle \mathcal{A} u, u - \bar{u} \rangle}{\|u\|_X} = \frac{\langle \mathcal{A}_1 u, u - \bar{u} \rangle + \langle \mathcal{A}_2 u, u - \bar{u}\rangle}{\|u\|_X}
        \geq \frac{\langle \mathcal{A}_2 u, u - \bar{u} \rangle}{\|u\|_X} - \|\mathcal{A}_1( \bar{u})\|_{X'} \left(1 + \frac{\| \bar{u}\|_X}{\|u\|_X}\right).
    \end{equation*}
    Since $\mathcal{A}_2$ is coercive and the negative terms on the right-hand side are bounded, it follows that $\mathcal{A}$ is coercive.
\end{proof}

\begin{theorem}[Abstract Existence Theorem]
\label{abstract existence theorem}
    Suppose further that $\mathbb{K}$ is closed. Let $\mathcal{A}: \mathbb{K} \rightarrow X'$ be a monotone, coercive, and hemicontinuous operator. Then, there exists $u \in \mathbb{K}$ such that
    \begin{equation*}
        \left\langle\mathcal{A}u, v - u\right\rangle \geq 0,
    \end{equation*}
    for all $v \in \mathbb{K}$.
\end{theorem}
\begin{proof}
See \cite{KiSt00}.
\end{proof}

\subsubsection{Functions of Bounded Variation}\label{subsec:BV_pre}
Let $U \subset \mathbb{R}^{n}$ be a bounded open set. We denote by $\mathcal{H}^{n-1}$ the $(n-1)$-dimensional Hausdorff measure.
If $U$ has a Lipschitz boundary $\partial U$, $\nu$ denotes the outward unit normal to $\partial U$, which exists $\mathcal{H}^{n-1}$-a.e.\ on $\partial U$.
\begin{definition}[Bounded Variation]\label{BV_functions}
     Let $f \in L^{1}\left(U\right)$. We say that $f$ is a function of bounded variation in $U$ if its distributional derivative is represented by an $\mathbb{R}^{n}$-valued Radon measure $Df = \left(D_{1}f, \ldots, D_{n}f \right)$ in $U$. That is,
\begin{equation*}
    \int_{U} f \frac{\partial \phi}{\partial x_{i}} \dx = -\int_{U} \phi D_{i}f \quad \text{for all } \phi \in \mathcal{C}^{\infty}_{c}\left(U\right), \quad i = 1, \ldots, n.
\end{equation*}
The class of all functions of bounded variation in $U$ forms a Banach space, $BV\left(U\right)$, with the norm
    \[\lVert f\rVert_{BV(U)} := \lVert f\rVert_{L^1(U)} + \left\lvert Df\right\rvert (U),\] where $\left\lvert Df\right\rvert$ denotes the 
    total variation of the measure $Df$. 
\end{definition}
\begin{theorem}\label{thm:gauss-green-bv}
    For a bounded open set $U$ with a Lipschitz boundary $\partial U$, there exists a unique, bounded linear mapping
    \begin{equation*}
        \tr: BV\left(U\right) \rightarrow L^{1}\left(\partial U; \mathcal{H}^{n-1}\right)
    \end{equation*}
    such that
    \begin{equation*}
        \int_{U} f \frac{\partial \phi}{\partial x_{i}} \dx = - \int_{U} \phi D_{i}f + \int_{\partial U} \left(\phi \nu_{i}\right) \tr f \dd \mathcal{H}^{n-1}, \quad i = 1, \dots, n,
    \end{equation*}
    for all $f \in BV(U)$ and $\phi \in \mathcal{C}^{1}(\mathbb{R}^{n})$.
\end{theorem}
\begin{proof}
    See \cite{evansgariepy2015}.
\end{proof}
\begin{definition}[Trace]
    Let $U$ be a bounded open set with a Lipschitz boundary $\partial U$.
    For $f\in BV(U)$, the function $\tr f \in L^1(\partial U)$ given by the above theorem is called the trace of $f$ on $\partial U$.
\end{definition}
\begin{theorem}[Compactness for $BV$ Functions]\label{compactness}
    Let $U$ be a bounded open set with a Lipschitz boundary $\partial U$. Assume $\left\{f_{k} \right\}_{k = 1}^{\infty}$ is a sequence in $BV\left(U\right)$ satisfying
    \begin{equation*}
       \sup_{k \in \mathbb{N}} \lVert f_k\rVert_{BV(U)} = \sup_{k \in \mathbb{N}}\left\{\int_{U} \left\lvert f_{k} \right\rvert \dx + \left\lvert Df_{k} \right\rvert\left(U\right) \right\} < \infty.
    \end{equation*}
    Then, there exists a subsequence $\left\{f_{k_{j}}\right\}_{j=1}^{\infty}$ and $f \in BV\left(U\right)$ such that $f_{k_{j}} \cws f$ in $BV(U)$ as $j \rightarrow \infty$. That is,
    \begin{align*}
        f_{k_{j}} &\rightarrow f \quad \text{in } L^{1}\left(U\right), \\
        Df_{k_{j}} &\cws Df \quad \text{in the sense of measures}.
    \end{align*}
\end{theorem}
\begin{proof}
    See \cite{AFP2000}, Chapter 3.
\end{proof}

\begin{lemma}\label{lem:limit_bv}
Let $U$ be a bounded open set with a Lipschitz boundary $\partial U$. Let $f_{k} \cws f$ in $BV(U)$ such that $f_{k} \in W^{1,1}_0(U)$. Let $\phi^{k} \rightarrow \phi$ uniformly in $\mathcal{C}(\overline{U})$. Then
\[ \lim_{k \to \infty} \int_U \phi^{k} D_{i}f_{k} = \int_U \phi D_{i}f - \int_{\partial U} (\phi \nu_{i}) \tr f \dd\mathcal{H}^{n-1} \quad i = 1, \dots, n. \]
\end{lemma}
\begin{proof}

    Consider the bounded linear functional $\Lambda$ on $\mathcal{C}(\overline{U})$ defined by
    \begin{equation*}
        \Lambda(\eta) := \int_U \eta D_i f - \int_{\partial U} (\eta \nu_i) \tr f \dd \mathcal{H}^{n-1},
    \end{equation*}
    for any $\eta \in \mathcal{C}(\overline{U})$.

    Fix $\epsilon > 0$. Since $\mathcal{C}^1(\overline{U})$ is dense in $\mathcal{C}(\overline{U})$, there exists a function $\Phi_\epsilon \in \mathcal{C}^1(\overline{U})$ such that $\|\Phi_\epsilon - \phi\|_{L^{\infty}} < \epsilon$. Since $f_k \in W^{1,1}_0(U)$, we have $f_k = 0$ on $\partial U$. Integration by parts gives
    \begin{equation*}
        \int_U \Phi_\epsilon D_i f_k = - \int_U f_k \frac{\partial \Phi_\epsilon}{\partial x_i} \dx.
    \end{equation*}
    Observe that we can pass to the limit on the right-hand side
    \begin{equation*}
        \lim_{k \to \infty} \int_U \Phi_\epsilon D_i f_k = - \int_U f \frac{\partial \Phi_\epsilon}{\partial x_i} \dx,
    \end{equation*}
    since $f_k \to f$ in $L^1(U)$ and $\frac{\partial \Phi_\epsilon}{\partial x_i} \in L^\infty(U)$.
    Using the Gauss-Green formula (Theorem~\ref{thm:gauss-green-bv}) for the $BV$ function $f$, we identify this limit as
    \begin{equation} \label{eq:smooth_limit}
        - \int_U f \frac{\partial \Phi_\epsilon}{\partial x_i} \dx = \int_U \Phi_\epsilon D_i f - \int_{\partial U} (\Phi_\epsilon \nu_i) \tr f \dd \mathcal{H}^{n-1} = \Lambda(\Phi_\epsilon).
    \end{equation}

    Now, we estimate the difference for the sequence $\phi^k$, using the uniform bound on total variation $M := \sup_k |Df_k|(U) < \infty$. Observe that we have
    \begin{align*}
        \left| \int_U \phi^k D_i f_k - \Lambda(\phi) \right| &\leq \left| \int_U (\phi^k - \Phi_\epsilon) D_i f_k \right| + \left| \int_U \Phi_\epsilon D_i f_k - \Lambda(\Phi_\epsilon) \right|
         + \left| \Lambda(\Phi_\epsilon) - \Lambda(\phi) \right| \\
         &\leq M \|\phi^k - \Phi_\epsilon\|_{L^\infty} + \left| \int_U \Phi_\epsilon D_i f_k - \Lambda(\Phi_\epsilon) \right| + C \epsilon,
    \end{align*}
where $C$ depends on $|Df|(U)$ and $\|\tr f\|_{L^1(\partial U)}$. For $k$ large enough, $\|\phi^k - \Phi_\epsilon\|_{L^\infty} \le 2\epsilon$ and the second term vanishes as $k \to \infty$.
Taking the limit superior as $k \to \infty$ in the inequality yields
    \begin{equation*}
        \limsup_{k \to \infty} \left| \int_U \phi^k D_i f_k - \Lambda(\phi) \right| \leq (2M + C)\epsilon.
    \end{equation*}
    Since $\epsilon > 0$ was arbitrary, letting $\epsilon \to 0$ proves that the limit is zero.
\end{proof}

\subsubsection{Growth Assumptions}
 For the existence analysis, we impose the following growth conditions on the Lagrangian~$L$.
\begin{hyp}[Growth Conditions on the Lagrangian]\label{assump:L-growth}
    There exist constants $c > 0$ and $l > 1$ such that
    \begin{equation*}
        \left\lvert D_{v}L\left( x,v\right) \right\rvert \leq c\lvert v \rvert^{l-1} + c,
    \end{equation*}
    for all $\left( x,v\right) \in \left[0,1\right] \times \mathbb{R}$.
\end{hyp}
\begin{remark}\label{upp_bnd_L}
The upper bound on $D_{v}L$ implies an upper bound on $L$ given by
\begin{equation*}
    \left\lvert L\left( x,v\right) \right\rvert \leq c\lvert v \rvert^{l} + c.
\end{equation*}
\end{remark}
\begin{hyp}[Lower bound on the Lagrangian]\label{assump:low_bnd_L}
    There exist constants $c > 0$ and $\kappa > 1$ such that
    \begin{equation*}
        L\left( x,v\right) \geq \frac{1}{c} \left\lvert v\right\rvert^{\kappa} - c,
    \end{equation*}
    for all $\left( x,v\right) \in \left[0,1\right] \times \mathbb{R}$.
\end{hyp}

\subsection{Operator Formulation}\label{existence}
First, we define the operators that will be studied. We choose the regularization exponent $q$ such that
\begin{equation}\label{relation:q_and_l}
    q \geq l+1,
\end{equation}
where $l$ is the growth parameter of the Lagrangian $L$ in Assumption~\ref{assump:L-growth}.
This choice ensures that the regularized operator is well-defined and coercive (see Propositions~\ref{p and l} and~\ref{Coercivity theorem}). Moreover, let $q'$ be the conjugate exponent such that $\frac{1}{q} + \frac{1}{q'} = 1$.

We define the domains
$\mathcal{S}^{+}\subset \mathcal{C}^{1}\left(\bar{\Omega}\right)$,
$\mathcal{M}_{0}^{+}\subset W^{1,q}_{0}\left(\Omega \right)$, and $\mathcal{B}^{t}_{0} \subset BV \left(\Omega \right)$, respectively, as follows
\begin{equation}\label{admissible_sets}
\begin{aligned}
    \mathcal{S}^{+} &= \left\{w \in \mathcal{C}^{1}(\bar{\Omega}) : w = 0 \text{ on } \partial \Omega, \ w_{x} + \varphi_{x}^{0} > 0 \right\},\\
    \mathcal{M}_{0}^{+} &= \left\{w \in W^{1,q}_{0}\left(\Omega \right): w_{x} + \varphi_{x}^{0} \geq 0 \right\}, \\
    \mathcal{B}^{t}_{0} &= \left\{w \in BV \left(\Omega \right): w(0, x) = w(T, x) = 0 \text{ in the sense of traces}\right\}.
\end{aligned}
\end{equation}
We impose $w=0$ on $\partial\Omega$ in $\mathcal{S}^+$ (and in $W^{1,q}_0(\Omega)$) to encode the homogenized Dirichlet condition $\psi=0$ on $\partial\Omega$; in contrast, $\mathcal{B}_0^t$ only prescribes the temporal traces, while the lateral boundary conditions are enforced through the boundary terms in the variational inequality.

\begin{proposition}
    The set $\mathcal{M}_{0}^{+}$ is a closed non-empty convex set.
\end{proposition}
\begin{proof}
The set $\mathcal{M}_{0}^{+}$ is non-empty, as it contains the zero function $w=0$ (since $\varphi_{x}^{0} \geq 0$).
To prove that $\mathcal{M}_{0}^{+}$ is closed, let $\{w^{n}\} \subset \mathcal{M}_{0}^{+}$ be such that $w^{n} \rightarrow w$ strongly in $W^{1,q}_{0}(\Omega)$ as $n \rightarrow \infty$.
Strong convergence implies that $Dw^{n} \to Dw$ in $L^q(\Omega)$. Thus, there exists a relabeled subsequence such that $w^{n}_{x} \to w_{x}$ almost everywhere. Therefore, almost everywhere, we have
\begin{equation*}
    w_{x} + \varphi_{x}^{0} = \lim_{n \rightarrow \infty} \left(w^{n}_{x} + \varphi_{x}^{0}\right) \geq 0.
\end{equation*}
To see that $\mathcal{M}_{0}^{+}$ is convex, let $w^{1}, w^{2} \in \mathcal{M}_{0}^{+}$ and $\lambda \in [0, 1]$.
By linearity, we have
\begin{align*}
    \left( \lambda w^{1} + \left(1-\lambda\right)w^{2}\right)_{x} + \varphi_{x}^{0} &= \lambda w_{x}^{1} + \left(1-\lambda\right)w_{x}^{2} + \varphi_{x}^{0} \\
    & = \lambda w_{x}^{1} + \lambda \varphi_{x}^{0} + \left(1-\lambda\right)w_{x}^{2} + \left(1 - \lambda\right)\varphi_{x}^{0} \\
    &\geq 0.
\end{align*}
Therefore, $\lambda w^{1} + \left(1-\lambda\right)w^{2} \in \mathcal{M}_{0}^{+}$.
Hence, $\mathcal{M}_{0}^{+}$ is a non-empty closed convex set, as desired.
\end{proof}

To simplify the signs in our calculations, we define
\begin{equation}\label{F transformation}
    \tilde{F}\left(t, x, j, m\right) = F\left( x,-j - \varphi_{t}^{0},m+\varphi_{x}^{0}\right).
\end{equation}
The partial derivatives of $\tilde{F}$ are
\begin{equation}\label{par. der. F}
\begin{aligned}
    \tilde{F}_j(t, x, j, m) &= -F_j(x, -j - \varphi^{0}_{t}, m + \varphi^{0}_{x}), \\ \tilde{F}_m(t, x, j, m) &= F_m(x, -j - \varphi^{0}_{t}, m + \varphi^{0}_{x}).
\end{aligned}
\end{equation}
\begin{definition}[Associated Operator and its Regularization]
    We define the operator $A: \mathcal{S}^{+} \rightarrow \left(\mathcal{B}^{t}_{0}\right)'$ by
    \begin{multline}\label{operator}
        \left\langle A w, \zeta \right\rangle = \int_{\Omega} \tilde{F}_j\left(t, x,w_{t},w_{x}\right)D_{t}\zeta + \tilde{F}_m\left(t, x,w_{t},w_{x}\right)D_{x}\zeta + \left(w + \varphi^{0}\right)D_{x}\zeta  \\  + \int_0^T  \tilde{F}_m(t, 0, w_t, w_x)  \zeta(t,0) \dt - \int_0^T ( \tilde{F}_m(t, 1, w_t, w_x) + 1) \zeta(t,1) \dt
    \end{multline}
    for any $w \in \mathcal{S}^{+}$ and $\zeta \in \mathcal{B}^{t}_{0}$.

    Furthermore, for $0 < \varepsilon < 1$, we define the regularized operator $A_\varepsilon: \mathcal{M}_{0}^{+} \rightarrow \left(W^{1,q}_{0}\right)'$ by
    \begin{multline}\label{reg. operator}
        \left\langle A_{\varepsilon} w, \zeta \right\rangle = \int_{\Omega} \Bigl[ \tilde{F}_j\left(t, x,w_{t},w_{x} + \varepsilon \right)\zeta_{t} + \tilde{F}_m\left(t, x,w_{t},w_{x} + \varepsilon \right)\zeta_{x} \\
        + \left(w + \varphi^{0}\right)\zeta_{x}
        + \varepsilon \left( \left\lvert w\right\rvert^{q-2}w \zeta + \left\lvert Dw\right\rvert^{q-2} Dw \cdot D\zeta \right)\Bigr] \dd \Omega,
    \end{multline}
    for any $w \in \mathcal{M}^{+}_{0}$ and $\zeta \in W^{1,q}_{0}$.
\end{definition}

\begin{remark}
    The domain $\mathcal{M}_{0}^{+}$ is not suitable for the non-regularized operator $A$ because the partial derivatives $\tilde{F}_j$ and $\tilde{F}_m$ involve denominators that vanish if the total density is zero. Thus, $A$ is defined on $\mathcal{S}^{+}$, where the density is strictly positive (bounded away from zero on the compact set $\bar{\Omega}$).

    The regularizing term in~\eqref{reg. operator} corresponds to the weak form of the operator $-\Delta_{q} w + |w|^{q-2}w$. When tested against $w$, this term generates the full $W^{1,q}$-norm, yielding coercivity.
\end{remark}

In the next section, we apply Theorem~\ref{abstract existence theorem} to find a solution $\psi^\varepsilon \in \mathcal{M}_{0}^{+}$ that satisfies the regularized variational inequality
\begin{equation}\label{weak solution}
    \left\langle A_{\varepsilon} \psi^{\varepsilon}, w - \psi^{\varepsilon} \right\rangle \geq 0,
\end{equation}
for all $w \in \mathcal{M}_{0}^{+}$.

\subsection{Existence of Approximate Solutions}
To apply Theorem~\ref{abstract existence theorem},
we verify that the operator $A_\varepsilon$ is monotone, coercive, and hemicontinuous. 
To this end, we observe that $A_{\varepsilon}$ defined in~\eqref{reg. operator} decomposes as $A_\varepsilon = \mathcal{A}_{\tilde{F}} + \mathcal{A}_g + \mathcal{A}_{reg}$, where
\begin{equation}\label{Split_operator}
\begin{aligned}
    &\textbf{The perspective part $\mathcal{A}_{\tilde{F}}$:} \\
    & \quad\quad \left\langle \mathcal{A}_{\tilde{F}} w, \zeta \right\rangle = \int_{\Omega} \left[ \tilde{F}_j\left(t, x,w_{t},w_{x} + \varepsilon \right)\zeta_{t} + \tilde{F}_m\left(t, x,w_{t},w_{x} + \varepsilon \right)\zeta_{x} \right] \dd \Omega, \\
    &\textbf{The ranking part $\mathcal{A}_g$:} \\
    & \quad\quad \left\langle \mathcal{A}_g w, \zeta \right\rangle = \int_{\Omega} \left(w + \varphi^{0}\right)\zeta_{x} \dd \Omega, \\
    &\textbf{The regularizing part $\mathcal{A}_{reg}$:} \\
    & \quad\quad \left\langle \mathcal{A}_{reg} w, \zeta \right\rangle = \varepsilon \int_{\Omega}\left( \left\lvert w\right\rvert^{q-2}w \zeta + \left\lvert Dw\right\rvert^{q-2} Dw \cdot D\zeta \right) \dd \Omega.
    \end{aligned}
\end{equation}
When no ambiguity arises, we write $\int_{\Omega} \tilde{F}_j\dd\Omega$ and 
$\int_{\Omega} \tilde{F}_m\dd\Omega$ in place of 
$\int_{\Omega} \tilde{F}_j\!\left(t,x,w_{t},w_{x}+\varepsilon\right)\dd\Omega$ and 
$\int_{\Omega} \tilde{F}_m\!\left(t,x,w_{t},w_{x}+\varepsilon\right)\dd\Omega$, 
suppressing the dependence on $(t,x,w_t,w_x+\varepsilon)$.

We first verify that $A_\varepsilon$ is well-defined.
\begin{proposition}\label{p and l}
    Suppose Assumptions~\ref{assump:H}--\ref{assump:L-growth} hold and the relation~\eqref{relation:q_and_l} is satisfied. Then, the operator $A_{\varepsilon}$ is well-defined.
\end{proposition}
\begin{proof}
    Let $0<\varepsilon<1$ be fixed. We only need to prove that $\mathcal{A}_{\tilde{F}}$ is well-defined, since well-definedness for $\mathcal{A}_g$ and $\mathcal{A}_{reg}$ is immediate.
    By Young's inequality, we have the following
    \begin{align*}
        \left|\int_{\Omega}\tilde{F}_j \zeta_{t}\dd \Omega \right|
        &\leq \int_{\Omega} \left\lvert \tilde{F}_j \zeta_{t} \right\rvert \dd \Omega \\
        &\leq \frac{q-1}{q} \int_{\Omega}\left\lvert \tilde{F}_j\right\rvert^{\frac{q}{q-1}}\dd \Omega + \frac{1}{q} \int_{\Omega}\left\lvert \zeta_{t}\right\rvert^{q}\dd \Omega =: \mathrm{I} + \mathrm{II}.
    \end{align*}
    Observe that since $w \in \mathcal{M}_0^+$, we have $w_x + \varphi_x^0 \geq 0$.
    By~\eqref{par. der. F} and~\eqref{eq:pers-derivatives} together with Assumption~\ref{assump:L-growth}, we observe that
    \begin{align*}
        \mathrm{I} = \int_{\Omega} \left\lvert \tilde{F}_j\right\rvert^{\frac{q}{q-1}}\dd \Omega &\leq C \int_{\Omega} \left( \left\lvert \frac{w_{t} + \varphi^{0}_{t}}{w_{x} + \varphi^{0}_{x} + \varepsilon}\right\rvert^{l-1} \right)^{\frac{q}{q-1}}
        \dd \Omega + C \\ 
        &\leq C(\varepsilon) \int_{\Omega} \left\lvert w_{t} + \varphi^{0}_{t} \right\rvert^{(l-1)\frac{q}{q-1}} \dd \Omega + C.
    \end{align*}
    Therefore, we have
    \begin{equation}\label{bound tilde{F}_j}
    \left|\int_{\Omega}\tilde{F}_j \zeta_{t}\dd \Omega\right|  \leq C (\varepsilon) \int_{\Omega} \left\lvert w_{t} + \varphi^{0}_{t} \right\rvert^{(l-1)\frac{q}{q-1}} \dd \Omega + \frac{1}{q} \int_{\Omega}\left\lvert \zeta_{t} \right\rvert^{q}\dd \Omega + C.
    \end{equation}
    Since $q \geq l+1 > l$, we have $(l-1)\frac{q}{q-1} < q$, thus the integral is well-defined.
    Similarly, we have by~\eqref{par. der. F},~\eqref{eq:pers-derivatives}, Assumption~\ref{assump:L-growth}, and Remark~\ref{upp_bnd_L}
    \begin{align*}
        \int_{\Omega} \left\lvert \tilde{F}_m \right\rvert^{\frac{q}{q-1}}\dd \Omega &\leq C \int_{\Omega} \left(\left\lvert \frac{w_{t} + \varphi^{0}_{t}}{w_{x} + \varphi^{0}_{x} + \varepsilon}\right\rvert^{l} \right)^{\frac{q}{q-1}}
\dd \Omega + C \\ 
        &\leq C(\varepsilon) \int_{\Omega} \left\lvert w_{t} + \varphi^{0}_{t} \right\rvert^{l\frac{q}{q-1}} \dd \Omega + C.
    \end{align*}
    Therefore, we have
    \begin{equation}\label{bound tilde{F}_m}
    \left|\int_{\Omega}\tilde{F}_m \zeta_{x}\dd \Omega\right|  \leq C (\varepsilon) \int_{\Omega} \left\lvert w_{t} + \varphi^{0}_{t} \right\rvert^{l\frac{q}{q-1}} \dd \Omega + \frac{1}{q} \int_{\Omega}\left\lvert \zeta_{x} \right\rvert^{q}\dd \Omega + C.
    \end{equation}
    Since $q \geq l+1$, we have $l\frac{q}{q-1} \leq q$, so the integral is well-defined. This concludes the proof.
\end{proof}

Now, we show that the operator $A_{\varepsilon}$ is monotone and coercive.
\begin{proposition}\label{Monotonicity theorem}
    Suppose Assumptions~\ref{assump:H}-\ref{assump:L-growth} hold and the relation~\eqref{relation:q_and_l} is satisfied. Then $A_{\varepsilon}$ is monotone in $\mathcal{M}^{+}_{0}$.
\end{proposition}
\begin{proof}
We proceed by proving that each part of $A_{\varepsilon}$, given in~\eqref{Split_operator}, is monotone.
\begin{itemize}
    \item  \textbf{The perspective part $\mathcal{A}_{\tilde{F}}$}: Let the functional $\mathcal{J}: \mathcal{M}^{+}_{0} \to \mathbb{R}$ be defined by
    \begin{equation*}
        \mathcal{J}(w) = \int_{\Omega} \tilde{F}(t, x, w_t, w_x+\varepsilon) \dd \Omega.
    \end{equation*}
    Since $F$ is convex in its last two arguments as in Remark~\ref{rem:Lagrangian}, $\tilde{F}$ is also convex because it is obtained by a linear transformation of coordinates. Since the integral preserves convexity, $\mathcal{J}$ is a convex functional. Note that
    \begin{equation*}
         \langle \mathcal{A}_{\tilde{F}} w,\zeta\rangle = \left. \frac{d}{d\lambda} \mathcal{J}(w+\lambda\zeta)\right\rvert_{\lambda=0},
    \end{equation*}
    that is, $\mathcal{A}_{\tilde{F}}$ is the G\^{a}teaux derivative of $\mathcal{J}$. Hence, it is a monotone operator.
    \medskip

    \item \textbf{The ranking part $\mathcal{A}_g$}: To check for monotonicity, we compute
    \begin{equation*}
    \left\langle \mathcal{A}_g(w^1) - \mathcal{A}_g(w^2), w^1-w^2 \right\rangle = \frac{1}{2} \int_{\Omega} \frac{\partial}{\partial x} (w^1 - w^2)^2 \dd \Omega = 0,
    \end{equation*}
    since $w^1-w^2 \in W_{0}^{1,q}(\Omega)$ and thus has zero trace on the boundary.
    \medskip

    \item \textbf{The regularizing part $\mathcal{A}_{reg}$:} The term
    \begin{equation*}
         \zeta\mapsto \enskip \varepsilon \left(\left\lvert w\right\rvert^{q-2}w\zeta+ \left\lvert Dw\right\rvert^{q-2} Dw \cdot D\zeta \right),
    \end{equation*}
    is the G\^{a}teaux derivative of the strictly convex functional
    \begin{equation*}
        \frac{\varepsilon}{q} \left( \|w\|_{L^q(\Omega)}^q + \|Dw\|_{L^q(\Omega)}^q \right).
    \end{equation*}
    Therefore, $\mathcal{A}_{reg}$ is a monotone operator (even strictly monotone).
\end{itemize}
Since $A_\varepsilon$ is the sum of three monotone operators, it is also monotone.
\end{proof}
\begin{proposition}\label{Coercivity theorem}
Suppose Assumptions~\ref{assump:H}--\ref{assump:L-growth} hold and the relation~\eqref{relation:q_and_l} is satisfied. Then $A_{\varepsilon}$ is coercive in $\mathcal{M}^{+}_{0}$.
\end{proposition}
\begin{proof}
    We use the equivalent $W^{1,q}_0(\Omega)$ norm
    $\|w\| := \left(\|w\|_{L^q(\Omega)}^q + \|Dw\|_{L^q(\Omega)}^q\right)^{1/q}$
    throughout this argument. Noting that $0\in \mathcal{M}^{+}_{0}$, we show that
    \begin{equation*}
        \frac{\left\langle A_{\varepsilon} w, w\right\rangle}{\left\lVert w \right\rVert} \rightarrow \infty, \quad \text{as} \quad \left\lVert w \right\rVert \rightarrow \infty.
    \end{equation*}
    Recall from~\eqref{Split_operator} that $A_\varepsilon = \mathcal{A}_{\tilde{F}} + \mathcal{A}_g + \mathcal{A}_{reg}$. By Proposition~\ref{Monotonicity theorem} and the fact that the sum of monotone operators is monotone, we find that $\mathcal{A}_{\tilde{F}} + \mathcal{A}_g$ is monotone.
    Observe that the operator $\mathcal{A}_{reg}$ is coercive. Indeed, by definition,
    \begin{equation*}
        \frac{\left\langle \mathcal{A}_{reg} w, w\right\rangle}{\left\lVert w \right\rVert} = \varepsilon \left[\int_{\Omega}\left( \left\lvert w\right\rvert^{q} + \left\lvert Dw\right\rvert^{q} \right) \dd \Omega\right]^{\frac{q-1}{q}} \rightarrow \infty, \quad \text{as} \quad \left\lVert w \right\rVert \rightarrow \infty.
    \end{equation*}
Therefore, by Lemma~\ref{coercivity_lemma}, $A_\varepsilon$ is coercive.
\end{proof}

\begin{proposition}\label{hemicontinuity}
    Suppose Assumptions~\ref{assump:H}--\ref{assump:L-growth} hold and the relation~\eqref{relation:q_and_l} is satisfied. Then, the operator $A_{\varepsilon}$ is hemicontinuous.
\end{proposition}
\begin{proof}
    For any $\tilde{w} \in \mathcal{M}^{+}_{0}$, let $z = \tilde{w} - w$ be the direction of perturbation. We define
    \begin{multline*}
        I(s) := \left\langle A_{\varepsilon}(w + s z), \zeta\right\rangle = \int_{\Omega} \tilde{F}_j\left(t, x, w_{t} + s z_{t}, w_{x} + s z_{x} + \varepsilon\right)\zeta_{t} \\ \quad + \tilde{F}_m\left(t, x, w_{t} + s z_{t}, w_{x} + s z_{x} + \varepsilon\right)\zeta_{x} \\ + (w + s z + \varphi^0 ) \zeta_x + \varepsilon \left( |w+sz|^{q-2}(w+sz) \zeta + |D(w+sz)|^{q-2} D(w+sz) \cdot D\zeta \right) \dd \Omega.
    \end{multline*}
    Applying the same analysis as in the proof of Proposition~\ref{p and l}, we deduce that $I(s)$ is continuous on $[0,1]$ by the Dominated Convergence Theorem, since the terms in the integrand can be bounded by an integrable function independent of $s$. Thus, $A_{\varepsilon}$ is hemicontinuous.
\end{proof}

Combining the preceding properties, we establish the existence of approximate solutions.
\begin{corollary}\label{approx_existence}
    Suppose Assumptions~\ref{assump:H}--\ref{assump:L-growth} hold and the relation~\eqref{relation:q_and_l} is satisfied. For any $0 < \varepsilon < 1$, there exists a solution $\psi^{\varepsilon} \in \mathcal{M}_{0}^{+}$ to the regularized variational inequality
    \begin{equation}\label{variational inequality}
        \left\langle A_{\varepsilon}\psi^{\varepsilon}, w - \psi^{\varepsilon}\right\rangle \geq 0, \quad \text{for all } w \in \mathcal{M}^{+}_{0}.
    \end{equation}
\end{corollary}
\begin{proof}
    The set $\mathcal{M}_0^+$ is a closed, non-empty, and convex subset of the reflexive Banach space $W^{1,q}_0(\Omega)$. By Propositions~\ref{p and l},~\ref{Monotonicity theorem}, \ref{Coercivity theorem}, and~\ref{hemicontinuity}, the operator $A_\varepsilon$ is well-defined, monotone, coercive, and hemicontinuous. Thus, the existence of $\psi^\varepsilon$ follows directly from Theorem~\ref{abstract existence theorem}.
\end{proof}

\subsection{A Priori Estimates and Convergence}
To pass to the limit as $\varepsilon \to 0$, we first show that $\lVert\psi^{\varepsilon}\rVert_{W^{1,1}_{0}}$ is bounded independently of $\varepsilon$. We begin by recording a consequence of Assumptions~\ref{assump:H} and~\ref{assump:positivity}.

\begin{remark}[Finite Energy of the Reference Path]\label{finite_energy}
    Under Assumptions~\ref{assump:H} and~\ref{assump:positivity}, we observe that the function $(t,x) \mapsto F(x, -\varphi^{0}_{t}(t,x), \varphi^{0}_{x}(t,x))$ is continuous on the compact set $\bar{\Omega}$. Consequently, it is bounded and there exists $\tilde{c} > 0$ such that 
    \begin{equation*}
        \int_{\Omega} F \left( x, -\varphi^{0}_{t},\varphi^{0}_{x}\right) \dd \Omega \leq \tilde{c}.
    \end{equation*}
\end{remark}
\begin{proposition}\label{Estimates}
Suppose Assumptions~\ref{assump:H}--\ref{assump:low_bnd_L} hold and the relation~\eqref{relation:q_and_l} is satisfied. Let $\psi^\varepsilon$ be a solution to the variational inequality~\eqref{variational inequality} for a given $\varepsilon > 0$. Then there exists a constant $C>0$, independent of $\varepsilon$, such that
\begin{equation*}
    \int_{\Omega} \frac{\left\lvert \psi^{\varepsilon}_{t} + \varphi^{0}_{t}\right\rvert^{\kappa}}{\left\lvert \psi^{\varepsilon}_{x} + \varphi^{0}_{x} + \varepsilon \right\rvert^{\kappa-1}} \dd \Omega \leq C \quad \text{and} \quad \left\lVert D\psi^{\varepsilon} \right\rVert _{L^{1}(\Omega)} \leq C.
\end{equation*}
\end{proposition}
\begin{proof}
    From the variational inequality~\eqref{variational inequality}, we have
    \begin{equation*}
        \left\langle A_{\varepsilon}\psi^{\varepsilon},\psi^{\varepsilon} \right\rangle \leq 0,
    \end{equation*}
    by setting the test function $w = 0$. That is, we have
   \begin{equation*}
        \int_{\Omega} \tilde{F}_j \psi^{\varepsilon}_{t} + \tilde{F}_m \psi^{\varepsilon}_{x} + \left(\psi^{\varepsilon} + \varphi^{0}\right)\psi^{\varepsilon}_{x}  \dd \Omega+ \varepsilon \left\lVert \psi^{\varepsilon} \right\rVert_{W^{1,q}_{0}}^{q} \leq 0.
   \end{equation*}
 We drop the non-negative regularization term $\varepsilon \left\lVert \psi^{\varepsilon} \right\rVert _{W^{1,q}_{0}}^{q}$ from the inequality.

\noindent\textbf{Step 1: Lower bound on the potential term.}
   We analyze the term involving the potential. We split the integral as
          \begin{align*}
              \int_{\Omega} \left(\psi^{\varepsilon} + \varphi^{0}\right)\psi^{\varepsilon}_{x} \dd \Omega = \int_{\Omega} \psi^{\varepsilon}\psi^{\varepsilon}_{x} \dd \Omega + \int_{\Omega}  \varphi^{0} \psi^{\varepsilon}_{x}\dd \Omega.
          \end{align*}
          By the Fundamental Theorem of Calculus and the boundary conditions, the first term vanishes
        \begin{align*}
                \int_{\Omega} \psi^{\varepsilon}\psi^{\varepsilon}_{x} \dd \Omega & =\int_{\Omega} \left(\frac{\left(\psi^{\varepsilon}\right)^2}{2}\right)_{x} \dd \Omega \\
                &= \int_{0}^{T} \left(\frac{\left(\psi^{\varepsilon}\right)^2}{2} \left(t,1\right)\right) - \left(\frac{\left(\psi^{\varepsilon}\right)^2}{2}(t,0)\right)\dt \\
                &= 0.
            \end{align*}
            For the second term, we add and subtract $\int_{\Omega} \varphi^0\varphi^0_x \dd \Omega$. Since $\psi^\varepsilon \in \mathcal{M}_0^+$, we have $\psi^\varepsilon_x + \varphi^0_x \geq 0$. Moreover, since $m_0, m_T > 0$, the interpolating potential satisfies $\varphi^0 \geq 0$. Thus, $\int_{\Omega} (\psi^\varepsilon_x + \varphi^0_x)\varphi^0 \dd \Omega \geq 0$, and we obtain
            \begin{align*}
                \int_{\Omega} \varphi^{0} \psi^{\varepsilon}_{x} +  \varphi^0\varphi^0_x -  \varphi^0\varphi^0_x\dd \Omega &= \int_{\Omega} (\psi^\varepsilon_x + \varphi^0_x)\varphi^0 \dd \Omega - \int_{\Omega}  \left(\frac{\left(\varphi^0\right)^2}{2}\right)_{x} \dd \Omega \\
                &\geq -\int_{0}^{T} \left[\frac{\left(\varphi^0\right)^2}{2} \left(t,1\right) - \frac{\left(\varphi^0\right)^2}{2}(t,0)\right]\dt \\
                &= -\frac{T}{2}.
            \end{align*}
        Combining these gives
            \begin{align*}
                \int_{\Omega} \left(\psi^{\varepsilon} + \varphi^{0}\right)\psi^{\varepsilon}_{x} \dd \Omega \geq -\frac{T}{2}.
            \end{align*}

\noindent\textbf{Step 2: Energy estimate.}
   Substituting this lower bound back into the main inequality gives
    \begin{equation*}
        \int_{\Omega} \tilde{F}_j \psi^{\varepsilon}_{t} + \tilde{F}_m \psi^{\varepsilon}_{x} \dd \Omega \leq \frac{T}{2}.
    \end{equation*}
    Adding $\int_{\Omega} \tilde{F}_j\varphi^{0}_{t}\dd \Omega$ and $\int_{\Omega} \tilde{F}_m \left(\varphi^{0}_{x}+\varepsilon\right)\dd \Omega$ to both sides yields
    \begin{align*}
        \int_{\Omega} \tilde{F}_j \left(\psi^{\varepsilon}_{t} + \varphi^{0}_{t}\right) + \tilde{F}_m \left( \psi^{\varepsilon}_{x} + \varphi^{0}_{x} + \varepsilon \right)\dd \Omega  &\leq \int_{\Omega} \left[ \tilde{F}_j \varphi^{0}_{t} + \tilde{F}_m \left(\varphi^{0}_{x} + \varepsilon \right) \right]\dd \Omega +\frac{T}{2} \\
        &\leq C\left(\tilde{c},T\right),
    \end{align*}
    where the second inequality follows from~\eqref{par. der. F}, Lemma~\ref{lem:pers-mu}, and Remark~\ref{finite_energy}. Note that because $\varphi^0_x(t,x) \geq \min_{y} \{m_0(y), m_T(y)\} > 0$, the mass argument in the perspective function remains strictly bounded away from zero, ensuring the bound $C(\tilde{c},T)$ is uniform in $\varepsilon$.

    We factor out $\psi^{\varepsilon}_{x} + \varphi^{0}_{x} + \varepsilon$ in the integrand of the left-hand side, to get
    \begin{equation*}
        \int_{\Omega} \left(\psi^{\varepsilon}_{x} + \varphi^{0}_{x} + \varepsilon\right) \left[ \tilde{F}_j \cdot \left(\frac{\psi^{\varepsilon}_{t} + \varphi^{0}_{t}}{\psi^{\varepsilon}_{x} + \varphi^{0}_{x} + \varepsilon}\right) + \tilde{F}_m \right] \dd \Omega \leq C\left(\tilde{c},T\right).
    \end{equation*}
    Observe that from~\eqref{par. der. F} and~\eqref{eq:pers-derivatives}, we have $\tilde{F}_j \cdot \left(\frac{\psi^{\varepsilon}_{t} + \varphi^{0}_{t}}{\psi^{\varepsilon}_{x} + \varphi^{0}_{x} + \varepsilon}\right) + \tilde{F}_m = L\left(x,-\frac{\psi^{\varepsilon}_{t} + \varphi^{0}_{t}}{\psi^{\varepsilon}_{x} + \varphi^{0}_{x} + \varepsilon}\right)$. Therefore, we bound the LHS from below by using Assumption~\ref{assump:low_bnd_L}, to get
    \begin{equation*}
        \int_{\Omega} \left(\psi^{\varepsilon}_{x} + \varphi^{0}_{x} + \varepsilon\right)\left[\frac{1}{c} \left\lvert\frac{\psi^{\varepsilon}_{t} + \varphi^{0}_{t}}{\psi^{\varepsilon}_{x} + \varphi^{0}_{x} + \varepsilon}\right\rvert^{\kappa} -c\right] \dd \Omega \leq C\left(\tilde{c},T\right).
    \end{equation*}

Rearranging terms, we conclude
    \begin{align*}
        \int_{\Omega} \frac{\left\lvert \psi^{\varepsilon}_{t} + \varphi^{0}_{t}\right\rvert^{\kappa}}{\left\lvert \psi^{\varepsilon}_{x} + \varphi^{0}_{x} + \varepsilon \right\rvert^{\kappa -1}} \dd \Omega &\leq C\left(\tilde{c},T\right) + c\int_{\Omega} \left(\psi^{\varepsilon}_{x} + \varphi^{0}_{x} + \varepsilon\right) \dd \Omega \\
        &\leq C\left(c,\tilde{c},T\right),
    \end{align*}
    since $\int_{\Omega} \left(\psi^{\varepsilon}_{x} + \varphi^{0}_{x} + \varepsilon\right) \dd \Omega \leq 2T$.

\noindent\textbf{Step 3: $L^1$ bounds.}
    The bound on $\|\psi_{x}^\varepsilon\|_{L^1}$ follows immediately from
    \begin{equation*}
        \int_{\Omega} \left\lvert \psi^{\varepsilon}_{x} \right\rvert \dd \Omega \leq \int_{\Omega} \left\lvert \psi^{\varepsilon}_{x} + \varphi^{0}_{x} \right\rvert \dd \Omega + \int_{\Omega} \left\lvert \varphi^{0}_{x} \right\rvert \dd \Omega \leq 2T.
    \end{equation*}
   For the time derivative, we use Young's inequality together with the triangle inequality

    \begin{align*}
        \int_{\Omega}\left\lvert \psi^{\varepsilon}_{t} \right\rvert \dd \Omega &\leq \int_{\Omega} \left[ \left\lvert \varphi^{0}_{t} \right\rvert + \left\lvert \psi^{\varepsilon}_{t} + \varphi^{0}_{t} \right\rvert \cdot \frac {\left\lvert \psi^{\varepsilon}_{x}+ \varphi^{0}_{x} + \varepsilon \right\rvert^{\frac{\kappa-1}{\kappa}}}{\left\lvert \psi^{\varepsilon}_{x} + \varphi^{0}_{x} + \varepsilon \right\rvert^{\frac{\kappa-1}{\kappa}}} \right] \dd \Omega \\
        &\leq \left\lVert \varphi^{0}_{t} \right\rVert_{L^{\infty}} T +\frac{1}{\kappa} \int_{\Omega} \frac{\left\lvert \psi^{\varepsilon}_{t} + \varphi^{0}_{t}\right\rvert ^{\kappa}}{\left\lvert \psi^{\varepsilon}_{x} + \varphi^{0}_{x} + \varepsilon\right\rvert^{\kappa-1}} \dd \Omega + \frac{\kappa-1}{\kappa} \int_{\Omega} \left\lvert \psi^{\varepsilon}_{x} + \varphi^{0}_{x} + \varepsilon \right\rvert \dd \Omega.
    \end{align*}
    Both integrals on the RHS are bounded by the previous steps. Thus, we get
    \begin{equation*}
        \left\lVert \psi^{\varepsilon}_{t} \right\rVert _{L^{1}} \leq C\left(\kappa, c, \tilde{c},T\right).\qedhere
    \end{equation*}
\end{proof}

\begin{corollary}[Compactness]\label{cor:compactness}
    Suppose Assumptions~\ref{assump:H}--\ref{assump:low_bnd_L} hold and the relation~\eqref{relation:q_and_l} is satisfied. Let $\{\psi^\varepsilon\}$ be a sequence of solutions to the variational inequality~\eqref{variational inequality}. Then, there exist a subsequence (still denoted $\psi^\varepsilon$) and a limit function $\psi \in BV(\Omega)$ such that $\psi^\varepsilon \to \psi$ strongly in $L^1(\Omega)$ and $D\psi^\varepsilon \cws D\psi$ in the sense of measures as $\varepsilon \to 0$. Moreover, $\psi$ satisfies $D_x\psi + \varphi_x^0 \geq 0$.
\end{corollary}
\begin{proof}
    We can see from the estimates in Proposition~\ref{Estimates} that any solution $\psi^{\varepsilon}$ to the variational inequality~\eqref{variational inequality} is uniformly bounded (with respect to $\varepsilon$) in $W^{1,1}_0\left( \Omega\right)$ and hence $\{\psi^\varepsilon\}$ is uniformly bounded in $BV(\Omega)$. Therefore, by Theorem~\ref{compactness}, we can extract a subsequence (still denoted $\psi^\varepsilon$) such that $\psi^\varepsilon \to \psi$ strongly in $L^1(\Omega)$ and $D\psi^\varepsilon \cws D\psi$ as measures. Finally, since $\psi^\varepsilon \in \mathcal{M}_0^+$, we have $\psi^\varepsilon_x + \varphi_x^0 \geq 0$ almost everywhere, 
which means $D_x\psi^\varepsilon + \varphi_x^0\dd x \geq 0$ as Radon measures. Since $D_x\psi^\varepsilon \cws D_x\psi$ 
in the sense of measures, and the weak$^*$ limit of non-negative Radon measures is non-negative, we conclude 
$D_x\psi + \varphi_x^0\dd x \geq 0$, i.e., $D_x\psi + \varphi_x^0 \geq 0$ in the sense of measures.
\end{proof}

Weak convergence in $BV$ does not preserve traces. Therefore, the 
boundary values of the limit function require separate analysis. 
The spatial boundary conditions are enforced through the variational inequality; for the temporal traces, the next proposition shows these are attained in the strong $L^1$ sense.

\begin{proposition}[Trace Attainment]\label{lem:temporal_trace}
    Suppose Assumptions~\ref{assump:H}--\ref{assump:low_bnd_L} hold and the relation~\eqref{relation:q_and_l} is satisfied. 
    Let $\psi^\varepsilon$ be the sequence of approximate solutions given by Corollary~\ref{approx_existence} and $\psi$ be the limit function given by Corollary~\ref{cor:compactness}. Then, in the sense of traces in $BV(\Omega)$, $\psi$ satisfies the boundary inequalities $\psi(t,0) \geq 0$ and $\psi(t,1) \leq 0$ for almost every $t \in (0,T)$, and the initial and terminal conditions $\psi(0, x) = \psi(T, x) = 0$ for almost every $x \in (0,1)$.
\end{proposition}
\begin{proof}
We first prove the spatial boundary inequalities.
Let $C = \|\varphi^0_x\|_{L^\infty(\Omega)}$. Since $\psi^\varepsilon \in \mathcal{M}_0^+$, for almost every $t \in (0,T)$ and $x \in (0,1)$, we have $ \psi_{x}^\varepsilon(t,x) \geq -C$. Furthermore, $\psi^\varepsilon(t,0) = 0$ for a.e. $t$.

By the Fundamental Theorem of Calculus with respect to the spatial variable $x$, for a.e.~$(t,x)$,
\begin{equation*}
    \psi^\varepsilon(t,x) = \int_0^x \psi^\varepsilon_x(t,y) \dy \geq -Cx.
\end{equation*}
Passing to the limit $\varepsilon \to 0$ (recalling $\psi^\varepsilon \to \psi$ in $L^1$), we have $\psi(t,x) \geq -Cx$ a.e. Taking the spatial trace (in $BV(\Omega)$) as $x \to 0$, we conclude $\psi(t,0) \geq 0$.
A symmetric argument integrating from $x$ to $1$ yields $\psi(t,1) \leq 0$, which concludes this part.

Now, for the temporal boundary conditions.
    We provide the proof for the terminal time $t=T$; the argument for $t=0$ is analogous. Consider a small time interval $[T-\delta, T]$. By Hölder's inequality with conjugates $\kappa$ and $\kappa' = \frac{\kappa}{\kappa-1}$, we estimate the $L^1$ norm of the time derivative $\psi_t^\varepsilon + \varphi^{0}_{t}$
    \begin{equation}
    \begin{aligned}
        \int_{T-\delta}^T \int_0^1 &|\psi_t^\varepsilon + \varphi^{0}_{t}| \dx\dt \\
        &= \int_{T-\delta}^T \int_0^1 \frac{|\psi_t^\varepsilon + \varphi^{0}_{t}|}{(\psi_x^\varepsilon + \varphi^{0}_{x} + \varepsilon)^{\frac{\kappa-1}{\kappa}}} (\psi_x^\varepsilon + \varphi^{0}_{x} + \varepsilon)^{\frac{\kappa-1}{\kappa}} \dx\dt \\
        &\leq \biggl( \int_{T-\delta}^T \int_0^1 \frac{|\psi_t^\varepsilon + \varphi^{0}_{t}|^\kappa}{(\psi_x^\varepsilon + \varphi^{0}_{x} + \varepsilon)^{\kappa-1}} \dx\dt \biggr)^{1/\kappa}
        \biggl( \int_{T-\delta}^T \int_0^1 (\psi_x^\varepsilon + \varphi^{0}_{x} + \varepsilon) \dx\dt \biggr)^{\frac{\kappa-1}{\kappa}}.
    \end{aligned}
    \end{equation}
    Using the energy bound from Proposition~\ref{Estimates}, the first integral is bounded by $C^{1/\kappa}$. The second integral represents the mass over the strip $[T-\delta, T] \times [0,1]$. Since $(\psi^\varepsilon + \varphi^{0})$ is monotone in $x$ with $(\psi^\varepsilon + \varphi^{0})(t, 1) - (\psi^\varepsilon + \varphi^{0})(t, 0) = 1$, we have $\int_0^1 (\psi_x^\varepsilon + \varphi^{0}_{x} + \varepsilon) \dx = 1 + \varepsilon$. Thus,
    \begin{equation}\label{eq:holder_time}
        \int_{T-\delta}^T \int_0^1 |\psi_t^\varepsilon + \varphi^{0}_{t}| \dx\dt \leq C^{1/\kappa} \bigl( \delta (1+\varepsilon) \bigr)^{\frac{\kappa-1}{\kappa}}.
    \end{equation}
    Passing to the limit as $\varepsilon \to 0$, and using the lower semicontinuity of the total variation, we obtain for the limit $\psi$
    \begin{equation}\label{eq:holder_time_lim}
        \left|D_t \psi + \varphi^{0}_{t} \dx\dt\right|\big([T-\delta,T]\times[0,1]\big) \leq C^{1/\kappa} \delta^{\frac{\kappa-1}{\kappa}}.
    \end{equation}
    We apply the Gauss-Green formula (Theorem~\ref{thm:gauss-green-bv}) on the rectangular domain $[t, T] \times [0, x]$ with the test function $\phi=1$ to obtain
\begin{equation*}
\int_0^x (\psi^\varepsilon + \varphi^{0})(T, y) \dy - \int_0^x (\psi^\varepsilon + \varphi^{0})(t, y)\dy = \int_t^T \int_0^x  (\psi_t^\varepsilon + \varphi_t^{0})(s, y) \dy \dd s,
\end{equation*}
and
\begin{equation*}
\int_0^x (\psi + \varphi^{0})(T, y) \dd y - \int_0^x (\psi + \varphi^{0})(t, y) \dd y = D_t\psi\big((t,T)\times(0,x)\big) + \int_t^T\int_0^x \varphi_t^{0}(s, y)\dy\dd s.
\end{equation*}
    Integrating the first equation over $t \in [T-\delta, T]$, taking absolute values, using~\eqref{eq:holder_time}, and using Fubini's theorem  (noting that $\psi^\varepsilon(T, \cdot) = 0$ in the sense of traces) gives   
    \begin{equation*}
        \biggl| \delta \int_0^x \varphi^0(T, y) \dd y - \int_{T-\delta}^T \int_0^x (\psi^\varepsilon + \varphi^{0})(t, y) \dd y\dt \biggr|
        \leq \delta \cdot C^{1/\kappa} \bigl(\delta(1+\varepsilon)\bigr)^{\frac{\kappa-1}{\kappa}}.
    \end{equation*}
    Passing $\varepsilon \to 0$ (recalling $\psi^\varepsilon \to \psi$ in $L^1$):
    \begin{equation}\label{Fub_seq}
        \biggl|\delta \int_0^x \varphi^0(T, y) \dy - \int_{T-\delta}^T \int_0^x (\psi + \varphi^{0})(t, y) \dy\dt \biggr| \leq \delta C^{1/\kappa} \delta^{\frac{\kappa-1}{\kappa}}.
    \end{equation}
    Similarly, repeating the process for the limit function $\psi$ directly from \eqref{eq:holder_time_lim}
    \begin{equation}\label{Fub_lim}
        \biggl| \delta \int_0^x (\psi+\varphi^0)(T, y) \dy - \int_{T-\delta}^T \int_0^x (\psi + \varphi^{0})(t, y) \dy\dt \biggr|
        \leq \delta C^{1/\kappa}\delta^{\frac{\kappa-1}{\kappa}}.
    \end{equation}
    By the triangle inequality, subtracting the two estimates \eqref{Fub_seq} and \eqref{Fub_lim} yields
    \begin{equation*}
        \delta \left\lvert\int_0^x(\psi+\varphi^0)(T, y) \dd y- \int_0^x \varphi^0(T, y)\dd y\right\rvert  \leq 2\delta C^{1/\kappa}\delta^{\frac{\kappa-1}{\kappa}}.
    \end{equation*}
    Dividing by $\delta$ and letting $\delta \to 0$ (since $1 - \frac{1}{\kappa} > 0$), the RHS vanishes. Thus, for any $x \in [0,1]$
    \begin{equation*}
        \int_0^x (\psi+\varphi^0)(T, y) \dy = \int_0^x \varphi^0(T, y) \dy.
    \end{equation*}
    Differentiating with respect to $x$, we conclude $(\psi+\varphi^0)(T, x) = \varphi^0(T, x)$ for almost every $x$. That is, $\psi(T,x) = 0$ for almost every $x \in (0,1)$.
\end{proof}

\begin{proposition}\label{existence of psi}
    Suppose Assumptions~\ref{assump:H}--\ref{assump:low_bnd_L} hold and the relation~\eqref{relation:q_and_l} is satisfied. Let $\psi^\varepsilon$ be the sequence of approximate solutions given by Corollary~\ref{approx_existence} and $\psi$ be the limit function given by Corollary~\ref{cor:compactness}. Then, for every test function $\tilde{\eta} \in \mathcal{S}^{+}$, given in~\eqref{admissible_sets}, the following variational inequality holds
\begin{multline*}
\int_{\Omega} \left[ \tilde{F}_j(t, x,\tilde{\eta}_t, \tilde{\eta}_x)\tilde{\eta}_t + \tilde{F}_m(t, x,\tilde{\eta}_t, \tilde{\eta}_x)\tilde{\eta}_x + \left(\tilde{\eta} +\varphi^0\right)\tilde{\eta}_x \right] \dd \Omega \\
- \biggl[\int_{\Omega} \tilde{F}_j(t, x,\tilde{\eta}_t, \tilde{\eta}_x) D_t\psi + \int_{\Omega} \tilde{F}_m(t, x,\tilde{\eta}_t, \tilde{\eta}_x) D_x\psi + \int_{\Omega} \left(\tilde{\eta} + \varphi^0\right) D_x\psi \\
+ \int_0^T ( \tilde{F}_m(t, 0, 0,
%\tilde{\eta}_t,
\tilde{\eta}_x) ) \psi(t,0) \dt - \int_0^T ( \tilde{F}_m(t, 1, 0,
%\tilde{\eta}_t,
\tilde{\eta}_x) + 1) \psi(t,1) \dt
\biggr] \geq 0,
\end{multline*}
\end{proposition}
\begin{proof}
We use Minty's method. Fix $\tilde\eta\in \mathcal{S}^+$.
Since $\psi^\varepsilon$ solves the regularized variational 
inequality~\eqref{variational inequality},
\[
\langle A_{\varepsilon} \psi^\varepsilon,\tilde\eta-\psi^\varepsilon\rangle \geq 0.
\]
By monotonicity of $A_\varepsilon$,
\[
\langle A_{\varepsilon} \tilde{\eta},\tilde{\eta} - \psi^{\varepsilon} \rangle
\geq \langle A_{\varepsilon} \psi^{\varepsilon} ,\tilde{\eta} - \psi^{\varepsilon} \rangle \geq 0.
\]
We consider the limit as $\varepsilon\to 0$. We show below that $\langle A_\varepsilon \tilde{\eta}, \tilde{\eta} - \psi^\varepsilon \rangle \to \langle A\tilde{\eta}, \tilde{\eta} - \psi \rangle$, from which the inequality $\langle A\tilde{\eta}, \tilde{\eta} - \psi \rangle \geq 0$ follows. By the definition of $A_\varepsilon$,
\begin{multline*}
\langle A_\varepsilon \tilde{\eta}, \tilde{\eta} - \psi^\varepsilon \rangle = \int_{\Omega} \tilde{F}_j(t, x,\tilde{\eta}_t, \tilde{\eta}_x+\varepsilon)\tilde{\eta}_t + \tilde{F}_m(t, x,\tilde{\eta}_t, \tilde{\eta}_x+\varepsilon)\tilde{\eta}_x + (\tilde{\eta} + \varphi^0)\tilde{\eta}_x \dd \Omega \\
- \int_{\Omega} \tilde{F}_j(t, x,\tilde{\eta}_t, \tilde{\eta}_x+\varepsilon) D_t\psi^\varepsilon - \int_{\Omega} \tilde{F}_m(t, x,\tilde{\eta}_t, \tilde{\eta}_x+\varepsilon) D_x\psi^\varepsilon - \int_{\Omega} (\tilde{\eta} + \varphi^0) D_x\psi^\varepsilon \\
 + \varepsilon \left( \int_{\Omega} \left\lvert \tilde{\eta} \right\rvert^{q} + \left\lvert D\tilde{\eta}\right\rvert^{q} - \left\lvert \tilde{\eta}\right\rvert^{q-2} \tilde{\eta} \psi^\varepsilon \dd \Omega - \int_{\Omega} \left\lvert D\tilde{\eta}\right\rvert^{q-2} D\tilde{\eta} \cdot D\psi^\varepsilon \right).
\end{multline*}
Similarly, the claimed limit is
\begin{multline*}
\langle A\tilde{\eta}, \tilde{\eta} - \psi \rangle = \int_{\Omega} \left[ \tilde{F}_j(t, x,\tilde{\eta}_t, \tilde{\eta}_x)\tilde{\eta}_t + \tilde{F}_m(t, x,\tilde{\eta}_t, \tilde{\eta}_x)\tilde{\eta}_x + (\tilde{\eta} + \varphi^0)\tilde{\eta}_x \right] \dd \Omega \\
- \int_{\Omega} \tilde{F}_j(t, x,\tilde{\eta}_t, \tilde{\eta}_x) D_t\psi - \int_{\Omega} \tilde{F}_m(t, x,\tilde{\eta}_t, \tilde{\eta}_x) D_x\psi - \int_{\Omega} (\tilde{\eta} + \varphi^0) D_x\psi \\ - \int_0^T ( \tilde{F}_m(t, 0, 0,
%\tilde{\eta}_t, 
\tilde{\eta}_x) ) \psi(t,0) \dt + \int_0^T ( \tilde{F}_m(t, 1, 
0,
%\tilde{\eta}_t,
\tilde{\eta}_x) + 1) \psi(t,1) \dt.
\end{multline*}

Observe that $(t, x,\tilde{\eta}_t\left(t,x\right), \tilde{\eta}_x\left(t,x\right)+\varepsilon) \rightarrow (t, x,\tilde{\eta}_t\left(t,x\right), \tilde{\eta}_x\left(t,x\right))$ uniformly on $\bar{\Omega}$, while staying inside the compact set, since $0<\varepsilon<1$,
\begin{equation}\label{eq:compact-set}
     \{(t,x,j,m)\colon (t,x)\in\bar{\Omega},\, j = \tilde{\eta}_t(t,x),\, \tilde{\eta}_x(t,x)\leq m \leq \tilde{\eta}_x(t,x)+1 \}.
\end{equation}
Because $\tilde{\eta} \in \mathcal{S}^+$, we have $\tilde{\eta}_x + \varphi^0_x > 0$. Thus, $m + \varphi^0_x \geq \tilde{\eta}_x + \varphi^0_x > 0$ on this set, which strictly avoids the zero-mass singularity of the perspective function $F$. 
Since $\tilde{F}_j$ and $\tilde{F}_m$ are uniformly continuous on the compact 
set~\eqref{eq:compact-set}, they preserve the uniform convergence.
Thus, we have
    \begin{equation*}
        \lim_{\varepsilon \rightarrow 0} \sup_{\left(t,x\right)} \lvert \tilde{F}_j(t, x,\tilde{\eta}_t, \tilde{\eta}_x+\varepsilon) - \tilde{F}_j(t, x,\tilde{\eta}_t, \tilde{\eta}_x)\rvert = 0,
    \end{equation*}
    and
    \begin{equation*}
        \lim_{\varepsilon \rightarrow 0} \sup_{\left(t,x\right)} \lvert \tilde{F}_m(t, x,\tilde{\eta}_t, \tilde{\eta}_x+\varepsilon) - \tilde{F}_m(t, x,\tilde{\eta}_t, \tilde{\eta}_x)\rvert = 0.
    \end{equation*}

By linearity, we write $\langle A_\varepsilon \tilde{\eta}, \tilde{\eta} - \psi^\varepsilon \rangle = \langle A_\varepsilon \tilde{\eta}, \tilde{\eta} \rangle - \langle A_\varepsilon \tilde{\eta}, \psi^\varepsilon \rangle$, and address the convergence in each term separately.
The first term, $\langle A_\varepsilon \tilde{\eta}, \tilde{\eta} \rangle$, is
\begin{equation*}
         \int_{\Omega} \left[ \tilde{F}_j(t, x,\tilde{\eta}_{t},\tilde{\eta}_{x} + \varepsilon) \tilde{\eta}_ {t} + \tilde{F}_m(t, x,\tilde{\eta}_{t},\tilde{\eta}_{x} + \varepsilon) \tilde{\eta}_ {x} + (\tilde{\eta} + \varphi^{0}) \tilde{\eta}_ {x} \right] \dd \Omega + \varepsilon \int_{\Omega} \left( \lvert \tilde{\eta}\rvert^{q} + \lvert D\tilde{\eta}\rvert^{q} \right) \dd \Omega.
\end{equation*}
    Taking the limit in the first integral, we have
    \begin{align*}
        \lim_{\varepsilon \rightarrow 0} \int_{\Omega} \tilde{F}_j(t, x,\tilde{\eta}_t, \tilde{\eta}_x+\varepsilon)\tilde{\eta}_t + \tilde{F}_m(t, x,\tilde{\eta}_t, \tilde{\eta}_x+\varepsilon)\tilde{\eta}_x + (\tilde{\eta} + \varphi^{0}) \tilde{\eta}_ {x} \dd \Omega \\\notag = \int_{\Omega} \tilde{F}_j(t, x,\tilde{\eta}_t, \tilde{\eta}_x)\tilde{\eta}_t + \tilde{F}_m(t, x,\tilde{\eta}_t, \tilde{\eta}_x)\tilde{\eta}_x + (\tilde{\eta} + \varphi^{0}) \tilde{\eta}_ {x}\dd \Omega.
    \end{align*}
    The second integral is $\varepsilon$ times a constant, hence converging to 0.
    The second term, $\langle A_\varepsilon \tilde{\eta}, \psi^\varepsilon \rangle$,
    reads
    \begin{align}
        \label{secondterm}
        \int_{\Omega} \tilde{F}_j(t, x,\tilde{\eta}_t, \tilde{\eta}_x+\varepsilon) D_t\psi^\varepsilon + \int_{\Omega} \tilde{F}_m(t, x,\tilde{\eta}_t, \tilde{\eta}_x+\varepsilon) D_x\psi^\varepsilon +\int_{\Omega} (\tilde{\eta} + \varphi^0) D_x\psi^\varepsilon \\\notag + \varepsilon  \int_{\Omega}  \left\lvert \tilde{\eta}\right\rvert^{q-2} \tilde{\eta} \psi^\varepsilon \dd \Omega + \varepsilon \int_{\Omega} \left\lvert D\tilde{\eta}\right\rvert^{q-2} D\tilde{\eta} \cdot D\psi^\varepsilon \dd \Omega
    \end{align}

    Since $D\psi^\varepsilon \cws D\psi$ in the sense of measures, we can pass to the limit in the linear term in~\eqref{secondterm} involving $D\psi^\varepsilon$.
    Specifically, the vector field $\Phi = (\tilde{F}_j, \tilde{F}_m + \tilde{\eta} + \varphi^0)$ converges uniformly to its limit (since $F$ is continuous and the arguments converge uniformly). Applying Lemma~\ref{lem:limit_bv} component-wise to the vector field $\Phi$, we obtain
    \begin{multline*}
         \lim_{\varepsilon \to 0} \langle A_\varepsilon \tilde{\eta}, \psi^\varepsilon \rangle
         = \int_{\Omega} \left[ \tilde{F}_j(t, x,\tilde{\eta}_t, \tilde{\eta}_x) D_{t}\psi + \tilde{F}_m(t, x,\tilde{\eta}_t, \tilde{\eta}_x) D_{x}\psi + (\tilde{\eta}+\varphi^{0}) D_{x}\psi \right] \\
         - \int_{\partial \Omega} \left[\tilde{F}_j(t, x,\tilde{\eta}_t, \tilde{\eta}_x)\nu_t + \left(\tilde{F}_m\left(t, x,\tilde{\eta}_t, \tilde{\eta}_x\right)+\tilde{\eta}+\varphi^{0}\right)\nu_x\right] \tr \psi \dd\mathcal{H}^{1}.
    \end{multline*}
By Proposition~\ref{lem:temporal_trace}, the trace $\tr \psi$ vanishes on the temporal boundaries ($t=0, T$) while $\nu_t=0$ on the rest of the boundary. Thus, the term $\tilde{F}_j \nu_t \tr \psi$ vanishes. 
Since $\tilde{\eta} \in \mathcal{S}^+$, we have $\tilde{\eta}(t,0) = \tilde{\eta}(t,1) = 0$, which also implies $\tilde{\eta}_t(t,0) = \tilde{\eta}_t(t,1) = 0$.
Together with $\varphi^0(t,0) = 0$ and $\varphi^0(t,1) = 1$, the boundary integrals simplify exactly to those required in the variational inequality:
\begin{equation*}
    \int_0^T \tilde{F}_m(t, 0, 0, \tilde{\eta}_x(t,0)) \psi(t,0) \dt - \int_0^T \left( \tilde{F}_m(t, 1, 0, \tilde{\eta}_x(t,1)) + 1 \right) \psi(t,1) \dt.
\end{equation*}
The terms involving $\varepsilon$ in~\eqref{secondterm} converge to zero because $\psi^\varepsilon$ is bounded in $L^1(\Omega)$ and $D\psi^\varepsilon$ is bounded in total variation, thanks to Proposition~\ref{Estimates}, while the coefficients involving $\tilde{\eta}$ are bounded uniformly.
Since $\lim_{\varepsilon \to 0} \langle A_\varepsilon \tilde{\eta}, \tilde{\eta} - \psi^\varepsilon \rangle 
= \langle A \tilde{\eta}, \tilde{\eta} - \psi \rangle$ and 
$\langle A_\varepsilon \tilde{\eta}, \tilde{\eta} - \psi^\varepsilon \rangle \geq 0$ 
for all $0 < \varepsilon < 1$, we conclude that $\langle A\tilde{\eta}, \tilde{\eta} - \psi \rangle \geq 0$.
\end{proof}

\begin{proof}[Proof of Theorem~\ref{Existence theorem}]
    Let $\psi^\varepsilon$ be the sequence of approximate solutions constructed in Proposition~\ref{Estimates}, and let $\psi \in BV(\Omega)$ be the limit function given by Corollary~\ref{cor:compactness}. We define $\varphi = \psi + \varphi^0$.

    First, we verify the properties of $\varphi$. Since $\psi \in BV(\Omega)$ and $\varphi^0 \in \mathcal{C}^2(\bar{\Omega})$, it follows that $\varphi \in BV(\Omega)$.
    By Proposition~\ref{lem:temporal_trace}, $\psi$ vanishes on the temporal boundaries $t=0$ and $t=T$ in the trace sense; therefore, $\varphi$ satisfies the initial and terminal conditions:
    \begin{equation*}
        \varphi(0, \cdot) = \varphi^0(0, \cdot) \quad \text{and} \quad \varphi(T, \cdot) = \varphi^0(T, \cdot).
    \end{equation*}
    Similarly, $\psi(t,0)\geq 0$ and $\psi(t,1)\leq 0$; therefore, $\varphi(t,0)\geq \varphi^0(t,0) = 0$ and $\varphi(t,1)\leq \varphi^0(t,1) = 1$. Furthermore, since $\psi^\varepsilon \in \mathcal{M}_0^+$, we have $ \psi_{x}^\varepsilon + \varphi_{x}^0 \geq 0$. Passing to the limit implies $D_x \psi + \varphi_{x}^0 \geq 0$ in the sense of measures, which means $D_x \varphi \geq 0$.

    Now, let $\eta$ be any test function satisfying the conditions of Definition~\ref{def:weak_solution}, i.e., $\eta \in \mathcal{C}^1(\bar{\Omega})$, $\eta_x > 0$, and $\eta = \varphi^0$ on $\partial\Omega$. 
    We define $\tilde{\eta} = \eta - \varphi^0$. Since $\eta = \varphi^0$ on the boundary, $\tilde{\eta} = 0$ on $\partial\Omega$. Moreover, $\tilde{\eta}_x + \varphi^0_x = \eta_x > 0$. Thus, $\tilde{\eta} \in \mathcal{S}^+$, i.e., it is a valid test function for the variational inequality established in Proposition~\ref{existence of psi}.
     Substituting $\psi = \varphi - \varphi^0$ and $\tilde{\eta} = \eta - \varphi^0$ into the inequality of Proposition~\ref{existence of psi}, and utilizing the definitions
\begin{equation*}
    \tilde{F}_j(t, x, \tilde{\eta}_t, \tilde{\eta}_x) = -F_j(x, -\eta_t, \eta_x) \quad \text{and} \quad \tilde{F}_m(t, x, \tilde{\eta}_t, \tilde{\eta}_x) = F_m(x, -\eta_t, \eta_x),
\end{equation*}
we find that the contributions involving $\varphi_t^0$ and $\varphi_x^0$ cancel exactly between the integrals upon substituting $\psi=\varphi-\varphi^0$. 

Moreover, since $\eta=\varphi^0$ on $\partial\Omega$, with $\varphi^0(t,0)=0$ and $\varphi^0(t,1)=1$, we have $\eta_t(t,0)=\eta_t(t,1)=0$. Recalling the derivative identity for the perspective function, we note that $F_m(x,0,m)=L(x,0)$. Hence, the boundary integrals simplify to:
\begin{align*}
&\int_0^T F_m\bigl(0,-\eta_t(t,0),\eta_x(t,0)\bigr)\,\varphi(t,0)\dt
 - \int_0^T \Big(F_m\bigl(1,-\eta_t(t,1),\eta_x(t,1)\bigr)+1\Big)\,(\varphi(t,1)-1)\dt \\
&=L(0,0)\int_0^T \varphi(t,0)\dt + \bigl(L(1,0)+1\bigr)\int_0^T (1-\varphi(t,1))\dt.
\end{align*}
This yields the exact variational inequality formulation of Definition~\ref{def:weak_solution}. Consequently, $\varphi$ is a weak solution to Problem~\ref{associated problem}.
\end{proof}

\end{document}